\documentclass{article}
\usepackage[]{graphicx}
\usepackage[]{amsmath,amssymb,amsthm}
\usepackage[inline,shortlabels]{enumitem}
\usepackage{verbatim,bm}
\usepackage{physics}
\usepackage[mathscr]{euscript}
\usepackage{longtable}
\usepackage{color}
\usepackage{xcolor}
\usepackage{comment}
\usepackage{float}
 
\usepackage[citation-order]{amsrefs}
\newenvironment{rezabib}
{\bibdiv\biblist\setupbib}
{\endbiblist\endbibdiv}
\def\setupbib{\catcode`@=\active}
\begingroup\lccode`~=`@
\lowercase{\endgroup\def~}#1#{\gatherkey{#1}}
\def\gatherkey#1#2{\gatherkeyaux{#1}#2\gatherkeyaux}
\def\gatherkeyaux#1#2,#3\gatherkeyaux{\bib{#2}{#1}{#3}}

\newtheorem{theorem}{Theorem}[section]
\newtheorem{lemma}[theorem]{Lemma}
\newtheorem{corollary}[theorem]{Corollary}
\theoremstyle{definition}

\newtheorem{example}[theorem]{Example}
\newcommand{\defn}[1]{{\em #1}}
\theoremstyle{remark}
\newtheorem{remark}[theorem]{Remark}

\newcommand{\Z}{\mathbb{Z}}
\newcommand{\N}{\mathbb{N}}
\newcommand{\Q}{\mathbb{Q}}
\newcommand{\NN}{\mathcal{N}}
\newcommand{\R}{\mathcal{R}}
\newcommand{\w}{\text{W}}
\newcommand{\bgw}{\text{BGW}}
\newcommand{\srg}{\text{SRG}}
\newcommand{\sgdd}{\text{SGDD}}
\newcommand{\ddg}{\text{DDG}}
\newcommand{\J}{\mathcal{J}}

\setlist[enumerate]{label=(\roman*)}

\def\0{{\bm 0}}   

\begin{document}
 \title{Quasi-balanced weighing matrices, signed strongly regular graphs and association schemes} 

\author{
 Hadi Kharaghani\thanks{Department of Mathematics and Computer Science, University of Lethbridge,
Lethbridge, Alberta, T1K 3M4, Canada. \texttt{kharaghani@uleth.ca}}
\and
Thomas Pender\thanks{Department of Mathematics and Computer Science, University of Lethbridge,
Lethbridge, Alberta, T1K 3M4, Canada. \texttt{thomas.pender@uleth.ca}}
\and
  Sho Suda\thanks{Department of Mathematics,  National Defense Academy of Japan, Yokosuka, Kanagawa 239-8686, Japan. \texttt{ssuda@nda.ac.jp}}
}
\date{\today}

\maketitle

\begin{abstract}
A weighing matrix $W$ is quasi-balanced if 
$|W||W|^\top=|W|^\top|W|$
has at most two off-diagonal entries, where $|W|_{ij}=|W_{ij}|$. A quasi-balanced weighing matrix $W$ signs a strongly regular graph if $|W|$ coincides with its adjacency matrix. 
Among other things, 
signed strongly regular graphs and their equivalent association schemes are presented.
\end{abstract}

\section{Introduction}

A {\it weighing matrix} of order $v$ and weight $k$, denoted $\w(v,k)$, is a $(0,\pm 1)$-matrix $W$ of order $v$ such that $WW^\top = kI$. The special cases in which $k=n-1$ and $k=n$ are termed the {\it Conference} and {\it Hadamard} matrices, respectively. 

A {\it unit-weighing matrix} $W$ of order $v$ and weight $k$ is a matrix with entries from $\{0\}\cup \{x \in \mathbb{C} \mid |x|=1\}$ of order $v$ such that $WW^*=k I_v$. 

Let $\R_n = \{e^{\frac{2\pi im}{n}} : 0 \leq m < n\}$ be the set of the complex $n$-th roots of unity, and let $W$ be a matrix of order $v$ over $\{0\} \cup \R_n$ such that $WW^* = kI_v$. We then say that $W$ is a \emph{Butson weighing} matrix of order $v$ and weight $k$ over $\R_n$.

A  weighing matrix is said to be {\it balanced} if after changing all the entries to their absolute values, it reduces to  the incidence matrix of a symmetric $(v,k,\lambda)$ balanced incomplete block design. In this case, the design is said to be {\it signed} and has the property of being {\it signable}.

The largest class of balanced weighing matrices are those with the classical parameters $\left(v=\frac{q^{m+1}-1}{q-1},k=q^m,\lambda=q^m-q^{m-1}\right)$ over cyclic groups, where $q$ is a prime power. Recently, it was shown in \cite{kps21} that the symmetric designs with parameters $\left(v=1+9\cdot\frac{9^{m+1}-1}{4},k=9^{m+1},\lambda=4\cdot 9^m\right)$ are signable.  
 
In this paper, we introduce a new class of weighing matrices which we will call {\it quasi-balanced}. For a quasi-balanced weighing matrix $W$, the absolute value matrix, i.e. the matrix whose entries are the absolute values of the entries in $W$, denoted $|W|$, enjoys the properties that $|W|$ and $|W|^\top$ commute, and the product $|W||W|^\top$ has at most two off-diagonal entries. Of particular interest is the case in which the matrix $|W|$ is highly structured, such as being the adjacency matrix of a strongly regular graph or the incidence matrix of a symmetric group divisible design.

A strongly regular graph with adjacency matrix $A$ is signable if some of the nonzero entries of $A$ can be changed to $-1$ such that the resulting matrix is a quasi-balanced weighing matrix. In our study of signable strongly regular graphs, we were led to consider the larger class of Butson weighing matrices. 

Finally, we study the association schemes obtained from signed strongly regular graphs or symmetric group divisible designs with certain block regularity conditions and the reverse implications. 
 
\section{Preliminaries}

Throughout this work, $I_n,J_n,O_n$ denote the identity matrix, all-ones matrix,  and zero matrix all of order $n$, respectively.  
We omit the subscript when the order is clear from the context. 

A {\it negacirculant matrix} of order $2n$ is a polynomial in the signed permutation (nega-shift) matrix of order $n$
$$
N=\begin{pmatrix}
0 & 1 & 0 & \cdots & 0 \\
0 & 0 & 1 & \cdots & 0 \\
\vdots & \vdots & \vdots & \ddots & \vdots \\
0 & 0 & 0 & \cdots & 1 \\
-1 & 0 & 0 & \cdots & 0 
\end{pmatrix}. 
$$ 
The matrix $N$ forms a multiplicative cyclic group of order $2n$ which we denote as $\NN_{2,n}$. Note that $g^{-1}=g^\top$ and $-g=g^n$, for every $g \in \NN_{2,n}$.

The idea of the previous paragraph can be generalized as follows. Let $\omega$ be a generator of $\R_u$. Then an $\omega$-circulant matrix is a polynomial in the monomial ($\omega$-shift) matrix of order $n$
\[
U=
 \begin{pmatrix}
  0 & 1 & 0 & \cdots & 0 \\
0 & 0 & 1 & \cdots & 0 \\
\vdots & \vdots & \vdots & \ddots & \vdots \\
0 & 0 & 0 & \cdots & 1 \\
\omega & 0 & 0 & \cdots & 0 
 \end{pmatrix}.
\]
As before, this generates a multiplicative cyclic group of order $\abs{\omega}n$ which we will denote as $\NN_{u,n}$. Note that $g^{-1}=g^*$ and $\omega g=g^n$, for every $g \in \NN_{u,n}$.

Let $\mathcal{G}=(V,E)$ be an undirected graph with $|V|=v$. The {\it adjacency matrix $A$} of $\mathcal{G}$ is a $v\times v$ $(0,1)$-matrix with rows and columns indexed by the elements of $V$ such that $A_{xy}=1$ if and only if $\{x,y\}\in E$. A graph is said to be a \emph{strongly regular graph} with parameters $(v,k,\lambda,\mu)$ if its adjacency matrix $A$ satisfies $A^2=k I_v+\lambda A+\mu(J_v-I_v-A)$. A graph is said to be a {\it Deza graph} with parameters $(v,k,b,a)$, $b>a$, if its adjacency matrix $A$ satisfies $A^2=k I_v+bX +a(J_v-I_v-X)$ for some symmetric $(0,1)$-matrix $X$ with diagonal entries $0$. Of course, a Deza graph for which $X=A$ is a strongly regular graph. Strongly regular graphs are studied in \cite{BVM}. Strongly regular graphs in the context of the more general distance regular graphs and along with their connections to the objects used in the essay are studied together in \cite{BCN,CVL,Tonchev}.

Let $m,n$ be positive integers both not smaller than two. A {\it (square) group divisible design} with parameters $(v,k,m,n,\lambda_1,\lambda_2)$ is a pair $(V,\mathcal{B})$, where $V$ is a finite set of  $v$ elements called points, and $\mathcal{B}$ a collection of $k$-element subsets of $V$ called blocks, where $|\mathcal{B}|=v$, and where the following are satisfied:
\begin{enumerate*}
 \item The point set $V$ can be partitioned into $m$ equinumerous classes of size $n$,
 \item distinct points from the same class are incident with $\lambda_1$ blocks, and
 \item points from distinct classes are incident with $\lambda_2$ blocks.
\end{enumerate*}
A group divisible design is said to be \emph{symmetric} (or to have the \emph{dual property}) if its dual, that is, the incidence structure obtained by interchanging the roles of the points and blocks is again a group divisible design with the same parameters. 
Throughout this paper, we always assume that $k>0$. 

The {\it incidence matrix} of a symmetric group divisible design is a $v\times v$ $(0,1)$-matrix $A$ with rows and columns indexed by the elements of $\mathcal{B}$ and $V$, respectively, such that, for $x\in V$ and $B\in \mathcal{B}$,
\begin{align*}
A_{B,x}=\begin{cases}
1 & \text{ if } x\in B,\\
0 & \text{ if } x\not\in B.
\end{cases}
\end{align*}
 After reordering the elements of $V$ and $\mathcal{B}$ appropriately,  
\begin{align}\label{eq:gdd}
A A^\top=A^\top A=k I_v+\lambda_1(\J_{m,n}-I_v)+\lambda_2(J_v-\J_{m,n}),  
\end{align}
where $\J_{m,n}=I_m\otimes J_n$. Then the $v\times v$ $(0,1)$-matrix $A$ is the incidence matrix of a symmetric group divisible design if and only if  the matrix $A$ satisfies Equation \eqref{eq:gdd}. Thus, we also refer to the $v\times v$ $(0,1)$-matrix $A$ satisfying  \eqref{eq:gdd} as a symmetric group divisible design. A {\it divisible design graph} is a graph such that its adjacency matrix is the incidence matrix of a symmetric group divisible design. The reader may consult the encyclopedic work \cite{BJL} for greater detail of divisible designs and related configurations.

Let $G$ be a multiplicatively written finite group not containing the symbol $0$. A {\it balanced generalized weighing matrix} with parameters $(v, k, \lambda)$ over $G$, denoted $\bgw(v, k, \lambda)$ over $G$, is a matrix $W$ of order $v$ with entries from $\{0\} \cup G$ such that
\begin{enumerate*}
 \item every row of $W$ contains exactly $k$ nonzero entries, and 
 \item for any distinct $i, h \in \{1, 2,\ldots, v\}$, every element of $G$ has exactly $\lambda/|G|$ copies in the multiset $\{W_{ij} W^{-1}_{hj} : 1 \leq j \leq v, W_{ij} \neq 0, W_{hj}\neq 0\}$.
\end{enumerate*}
A BGW with $k=\lambda=v$ is said to be a {\it generalized Hadamard matrix}, and a BGW with $k=v-1,\lambda=v-2$ is said to be a {\it generalized conference matrix}. A BGW, say $W$, over $C_{2n}=\langle g \rangle$, a cyclic group of order $2n$, is skew-symmetric if $W_{ji}=g^n W_{ij}$ whenever $i \neq j$ and $W_{ij} \neq 0$, and $W_{ii}=0$, for every $i$. For an overview of the theory of balanced generalized weighing matrices, please consult \cite{JK} and the references cited therein. The reader may also profitably consult the monograph \cite{IS} for a study of BGWs together with their relations to the objects studied here.

The largest class of known balanced generalized weighing matrices are those with the so-called classical parameters corresponding to the prime power $q$ $$\left(\frac{q^{m+1}-1}{q-1},q^m,q^m-q^{m-1}\right)$$  over a cyclic group $C_n$ in any of the following situations:
\begin{enumerate*}
 \item For any odd $q$, positive integer $m$, and a divisor $n$ of $q-1$;
 \item $q$ is even, $m$ is odd, and $n$ a divisor of $q-1$; and 
 \item $q$ and $m$ are even, and $n$ a divisor of $2(q-1)$.
\end{enumerate*}
Moreover, these matrices can be assumed to be $\omega$-circulant . Furthermore, for odd $q$, if the ratio $\frac{q-1}{n}$ is odd. the matrix is skew-symmetric; and if the ratio is even, the matrix is symmetric. See \cite{arasu,gm2,IS,jt,jt-ii} for the constructions of these objects.

We follow \cite{BI} as our reference for the remaining part of this section. A {\it symmetric association scheme} with {\it $d$-classes} and a finite vertex set $X$, is a set of non-zero $(0,1)$-matrices $A_0, A_1,\ldots, A_d$ with rows and columns indexed by $X$, such that
\begin{enumerate}
\item $A_0=I_{|X|}$;
\item $\sum_{i=0}^d A_i = J_{|X|}$;
\item $A_i^\top=A_i$ for $i\in\{1,\ldots,d\}$; and
\item for any $i$, $j$ and $k$, there exist non-negative integers $p_{i,j}^k$ such that $A_iA_j=\sum_{k=0}^d p_{i,j}^k A_k$.
\end{enumerate}
Note that a graph with adjacency matrix $A$ is strongly regular if and only if $\{I,A,J-I-A\}$ is a symmetric association scheme with $2$-classes. The \defn{intersection matrix} $B_i$ is defined by $(B_i)_{jk}=p_{i,j}^k$. Since each $A_i$ is symmetric, it follows from condition (iv) that the $A_i$'s necessarily commute. The linear space spanned by the $A_i$'s over the complex number field is closed under standard matrix multiplication and, therefore, forms a commutative algebra. We denote this algebra by $\mathcal{A}$, and it is called the \emph{Bose-Mesner algebra}. There exists also a basis of $\mathcal{A}$ consisting of primitive idempotents, say $E_0=(1/|X|)J_{|X|},E_1,\ldots,E_d$. Since  $\{A_0,A_1,\ldots,A_d\}$ and $\{E_0,E_1,\ldots,E_d\}$ are two bases of $\mathcal{A}$, there exist change-of-basis matrices $P$ and $Q$ such that
\begin{align*}
A_j=\sum_{i=0}^d P_{ij}E_i,\quad E_j =\frac{1}{|X|}\sum_{i=0}^{d} Q_{ij}A_i.
\end{align*}   
The matrices $P,Q$ are said to be the \defn{first and second eigenmatrices}, respectively.  

\section{Quasi-balanced weighing matrices} 

\subsection{Definition}
We now introduce the concept of a quasi-balanced weighing matrix for the first time. 
 A weighing matrix $W$ is said to be \emph{quasi-balanced} if the matrix $|W|$ defined by $|W|_{ij}=|W_{ij}|$ commutes with $|W|^\top$, and if the product $|W||W|^\top$ has at most two off-diagonal entries.

Whenever $|W||W|^\top$ has exactly one off-diagonal entry, the $(0,\pm 1)$-weighing matrix $W$ coincides with a balanced weighing matrix. 
Throughout this paper we restrict our interest to the  following four possibilities for the structure of $|W|$:
\begin{enumerate}
\item The adjacency matrix of a strongly regular graph with parameters $(v,k,\lambda,\mu)$. Then $|W|=A$ satisfies $A^\top=A$ and 
 $$
 A^2=k I_v+\lambda_1 A+\mu(J_v-I_v-A);
 $$   
 \item the adjacency matrix of a divisible design graph with parameters $(v,k,m,n,\lambda_1,\lambda_2)$. Then $|W|=A$ satisfies $A^\top=A$ and 
 $$
 A^2=(k-\lambda_1) I_v+(\lambda_1-\lambda_2)\J_{m,n}+\lambda_2 J_v;
 $$   
\item the incidence matrix of a symmetric group divisible design with parameters $(v,k,m,n,\lambda_1,\lambda_2)$. Then $|W|=A$ satisfies 
 $$
 AA^\top=A^\top A=(k-\lambda_1) I_v+(\lambda_1-\lambda_2)\J_{m,n}+\lambda_2 J_v \text{; and}
 $$  
 \item the adjacency matrix of a Deza graph with parameters $(v,k,b,a)$. Then $|W|=A$ satisfies $A^\top=A$ and 
 $$
A^2=k I_v+bX +a(J_v-I_v-X),  
 $$   
 for some symmetric $(0,1)$-matrix $X$ with diagonal entries $0$. 
 \end{enumerate}
 
The concept of \emph{Siamese objects} and the related graph decomposition will be touched upon in what follows. The reader is referred  to \cite{kt,krw} for the appropriate definitions.

\subsection{Constructions for signed strongly regular graphs}
It was shown in \cite{kt} that 
the complete graph $K_{1+q+q^2+q^3}$ on $1+q+q^2+q^3$ vertices can be decomposed into $1+q$ strongly regular graphs sharing $1+q^2$ cliques of size $1+q$. In this section it is shown that the decomposition is singable.
  We give the constructions according to the three cases $q \equiv -1 \text{ (mod $4$)}$, $q \equiv 1 \text{ (mod $4$)}$, and finally $q \equiv 0 \text{ (mod $2$)}$.

We begin with the case that $q \equiv -1 \pmod{4}$.

\begin{theorem}\label{thm:cons1}
Let $q$ be a prime power congruent to $-1\pmod{4}$. Then there is a quasi-balanced weighing matrix $\w(1+q+q^2+q^3,q+q^2)$, say $W$, for which $|W|$ is an
$$\srg(1+q+q^2+q^3,q+q^2,q-1,q+1).$$
Furthermore, $K_{1+q+q^2+q^3}$ decomposes into $1+q$ Siamese signed strongly regular graphs with the above parameters sharing $1+q^2$ cliques of size $1+q$.
\end{theorem}

\begin{proof}
Let $H$ be the skew-symmetric $\bgw(1+q^2,q^2,q^2-1)$ over the cyclic group $\NN_{2,1+q}$ of order $2+2q$ generated by the nega-shift matrix $N$ of order $1+q$. It follows from Section 3 of  \cite{jt} that there is a negacirculant $\w(q+1,q)$ with zero diagonal, say $C$. Let $W$ be the block matrix defined by

\begin{align*}
W_{ij} & =\begin{cases}
C & \text{if }i=j \text{, and} \\
  H_{ij}R &  \text{if } i \neq j,
\end{cases}
\end{align*}
where $R$ is the back identity matrix of order $1+q$. We claim that $W$ is a quasi-balanced weighing matrix $\w(1+q+q^2+q^3,q+q^2)$.
 
 To see this, let $r_1, \dots, r_v$ be the block rows of $W$. Then 
 $$
 r_ir_i^\top = CC^\top + \sum_{i=1}^{q^2} N^{\ell_i}R(N^{\ell_i}R)^\top,
 $$ 
 for some $\ell_1, \dots, \ell_{q^2} \in \{0,\dots,2q-1\}$. But $CC^\top=qI$ and $N^\ell R(N^\ell R)^\top=I$, for all $\ell$; hence, $r_ir_i^\top=(q+q^2)I$.
 
 Next, if $i<j$, then 
 \begin{align*}
 r_ir_j^\top &= C(H_{ji}R)^\top+H_{ij}RC^\top + \left( \frac{q-1}{2} \right)\sum_{g \in \NN_{2,1+q}}g \\
 &= -CH_{ij}R + CH_{ij}R + \left( \frac{q-1}{2} \right)\sum_{g \in \NN_{2,1+q}}g.
 \end{align*}
 Since the elements of $\NN_{2,1+q}$ sum to $O$, it follows that $r_ir_j^\top = O$.
 
 Next, we show that $\abs{W}$ is an SRG with the stated parameters, where it remains to show that $\lambda=q-1$ and $\mu=q+1$. Note that
 $$
 \abs{W}_{ij}=\begin{cases}
               J-I & \text{if $i=j$, and} \\
               \abs{H_{ij}}R & \text{if $i \neq j$.}
              \end{cases}
 $$
 Then $\abs{r_i}\abs{r_i}^T = (J-I)^2 + q^2I = (q+q^2)I + (q-1)(J-I)$. For $i<j$, we have that
 \begin{align*}
 \abs{r_i}\abs{r_j}^T &= (J-I)(\abs{H_{ji}}R)^T + (\abs{H_{ij}}R)(J-I)^T + (q-1)\sum_{\ell=0}^q \abs{N}^\ell \\
 &= (q+1)J - 2\abs{H_{ij}}R,
 \end{align*}
 which concludes the derivation.
 
 Finally, for $\ell \in \{0, \dots, q\}$, we define $W_\ell$  by
 $$
 (W_\ell)_{ij}=\begin{cases}
                C & \text{if } i=j \text{, and} \\
                H_{ij}N^\ell R & \text{if } i \neq j.
               \end{cases}
 $$
 Then $W=W_0, \dots, W_q$ is the required signed decomposition of $K_{1+q+q^2+q^3}$. This completes the proof. 
\end{proof}

\begin{example}
 A quasi-balanced $\w(40,12)$ is shown in Figure \ref{signed-w-40-12} of the appendix.
\end{example}

We now deal with the case where the prime power is  congruent to $1\pmod{4}$.

\begin{theorem}\label{thm:cons2}
Let $q$ be a prime power congruent to $1\pmod{4}$. Then there is a quasi-balanced Butson weighing matrix $\w(1+q+q^2+q^3,q+q^2)$ over $\R_4$, say $W$, for which $|W|$ is an
$$\srg(1+q+q^2+q^3,q+q^2,q-1,q+1).$$
Furthermore, the complete graph on $v$ vertices decomposes into $q+1$ Siamese strongly regular graphs signed over $\R_4$ with the above parameters all sharing $1+q^2$ cliques of size $q+1$. 
\end{theorem}

\begin{proof}
The construction is similar to the construction in Theorem~\ref{thm:cons1} with one difference. The BGW matrix $H$ in this case is symmetric, so we need to multiply the matrix $C$ by the complex $i$ in order to make the construction work. This  means that the matrix $W$ has entries in $\{0,\pm 1, \pm i\}$. The rest of the construction is the same. The graph decomposition property is also similar except that the shared part now has elements in $\{0,\pm i\}$.
\end{proof}

The final construction relates to the even prime powers. In what follows, let $\omega$ be a primitive complex $(q-1)$-st root of unity, and let $U$ be the $\omega$-shift matrix of order $1+q$, i.e. the generator of $\NN_{q-1,1+q}$.

\begin{theorem}\label{pp}
Let $q$ be an even prime power. Then there is a quasi-balanced weighing matrix $\w(1+q+q^2+q^3,1+q^3)$ of quaternions, say $W$, for which the off-diagonal entries of $\abs{W}\abs{W}^\top$ consist of the two values $q^3-q^2$ and $q^3-q^2+2$.
\end{theorem}

\begin{proof}
Let $H$ be the symmetric $\bgw(1+q^2,q^2,q^2-1)$ over the cyclic group $\NN_{q-1,1+q}$ of order $q^2-1$ generated by $U$ given above. By {\cite{jt-ii}}, there is an $\omega$-circulant $W(q+1,q)$ with zero diagonal, say $C$. Let $k$ be a quaternion unit $k^2=-1$, and let $W$ be the block matrix defined by
\begin{align*}
W_{ij} & =\begin{cases}
kI_q & \text{if } i=j \text{, and} \\  
 C H_{ij}R &  \text{if }i\ne j.
\end{cases}
\end{align*}

Let $r_1, \dots, r_v$ be the block rows of $W$. Then, since $(CH_{ij}R)(CH_{ij}R)^*=qI$, we have that $r_ir_i^*=(1+q^3)I$.

Next, noting that $XR=R\bar X$ and $Yk=k\bar{Y}$ where $X,Y\in\{g,C\}$ for every $g \in \NN_{q-1,1+q}$, it follows that $k(CH_{ij}R)^*+(CH{ij}R)k^*=k\bar{H_{ij}}{\bar{C}}R-k\bar{C}\bar{H{ij}}R=0.$ As before, then, it follows that $r_ir_j^*=O$.
\end{proof}

\begin{example}
Let $q=4$ in Theorem~\ref{thm:cons2}. Then $H$ is the symmetric $\bgw(17,16,15)$ over the group $G$ of order $15$ generated by the  $\omega$-shift matrix
$$U=\left(\begin{matrix} 0&1&0&0&0\\0&0&1&0&0\\0&0&0&1&0\\0&0&0&0&1\\\omega&0&0&0&0\end{matrix}\right),$$ where $\omega$ is a primitive $3$rd root of unity. The matrix $C$ can be taken to be
$$C=\left(\begin{matrix} 0 &\omega^2&\omega&\omega^2&\omega^2\\1 &0&\omega^2&\omega&\omega^2\\1&1&0&\omega^2&\omega\\\omega^2&1&1&0&\omega^2\\1&\omega^2&1&1&0\end{matrix}\right).$$
The matrix $W$ is then a quasi-balanced weighing matrix $\w(85,65)$ of quaternions, and the off-diagonal entries of $\abs{W}\abs{W}^\top$ consist of the two values $48$ and $50$.
\end{example}

\begin{remark} The constructed matrices in Theorem \ref{pp} having constant diagonal can be considered as the adjacency matrices of an $\srg(1+q+q^2+q^3,q^3,q^3-q^2,q^3-q^2)$ upon setting the nonzero entries to $1$ and the diagonal to $0$.
\end{remark}

\subsection{Constructions for signed symmetric group divisible designs}\label{sec:GHC}

We now present several constructions of signed divisible designs.

\begin{theorem}\label{thm:qbsgdd1}
If there exist a $\bgw(m,m,m)$ over the cyclic group of order $2n$  and a $\bgw(n,k,\lambda)$ over $\{1,-1\}$, then there exists a quasi-balanced weighing matrix $\w(mn,mk)$, say $W$, for which $|W|$ is an
$$
\sgdd\left(mn,mk,m,n,m\lambda,\frac{m}{n}k^2\right)
$$
 and satisfies the property that $|W|\J_{m,n}=\J_{m,n}|W|=kJ_{mn}$. 
\end{theorem}

\begin{proof}
Let $C_1$ be a $\bgw(m,m,m)$ over the cyclic group of order $2n$ generated by the negacirculant matrix $N$ of order $n$, and $C_2$ be a $\bgw(n,k,\lambda)$ over $\{1,-1\}$. For the matrix $C_1$, write the $(i,j)$-entry of $C_1$ as $N^{c_{ij}}$ for $i,j\in\{1,\ldots,m\}$. 

Define an $m \times m$ $(0,\pm1)$ block matrix $W$ by $W_{ij}=N^{c_{ij}}C_2$. Then $W$ is a quasi-balanced weighing matrix. Indeed, for $i,j\in\{1,\ldots,m\}$, 
\begin{align*}
(WW^\top)_{ij}&=\sum_{1\leq \ell \leq m} (N^{c_{i\ell}}C_2)(N^{c_{j\ell}}C_2)^\top\\
&=\sum_{1\leq \ell \leq m} N^{c_{i\ell}}C_2C_2^\top N^{-c_{j\ell}}\\
&=k\sum_{1\leq \ell \leq m} N^{c_{i\ell}-c_{j\ell}} \\
&=\begin{cases}
m kI_n & \text{if $i=j$, and} \\
O & \text{if }i\neq j,  
\end{cases}
\end{align*}
which shows that $W$ is a weighing matrix of order $mn$ and weight $mk$.  
Since $W_{ij}=N^{c_{ij}}C_2$, $|W|_{ij}=P^{c_{ij}}|C_2|$ where $P$ is the circulant  matrix with first row $(010\cdots0)$ of order $n$, we have
\begin{align*}
(|W||W|^\top)_{ij}&=\sum_{1\leq \ell \leq m} (P^{c_{i\ell}}|C_2|)(P^{c_{j\ell}}|C_2|)^\top\\
&=\sum_{1\leq \ell \leq m} P^{c_{i\ell}}|C_2|\cdot|C_2|^\top P^{-c_{j\ell}}\\
&=\sum_{1\leq \ell \leq m} P^{c_{i\ell}}((k-\lambda)I+\lambda J) P^{-c_{j\ell}}\\
&=(k-\lambda)\sum_{1\leq \ell \leq m} P^{c_{i\ell}-c_{j\ell}}+\lambda \sum_{1\leq \ell \leq m} P^{c_{i\ell}} J P^{-c_{j\ell}}\\
&=(k-\lambda)\sum_{1\leq \ell \leq m} P^{c_{i\ell}-c_{j\ell}}+\lambda \sum_{1\leq \ell \leq m}J \\
&=\begin{cases}
m((k-\lambda)I+\lambda J) & \text{if $i=j$, and} \\
(\frac{m}{n}(k-\lambda)+m\lambda) J & \text{if }i\neq j.  
\end{cases}
\end{align*}
On the other hand,  
\begin{align*}
(|W|^\top|W|)_{ij}&=\sum_{1\leq \ell \leq m} (P^{c_{i\ell}}|C_2|)^\top (P^{c_{j\ell}}|C_2|)\\
&=\sum_{1\leq \ell \leq m} |C_2|^\top P^{-c_{i\ell}+c_{j\ell}}|C_2|\\
&=\begin{cases}
m|C_2|^\top|C_2| & \text{if $i=j$, and} \\
\frac{m}{n}|C_2|^\top J |C_2| & \text{if }i\neq j. 
\end{cases}\\
&=\begin{cases}
m((k-\lambda)I+\lambda J) & \text{if $i=j$, and} \\
\frac{m}{n}k^2 J & \text{if }i\neq j.   
\end{cases}
\end{align*}
By the equality $k(k-1)=\lambda(n-1)$, we have $|W||W|^\top=|W|^\top |W|$. 
Therefore $|W|$ is a symmetric group divisible design with parameters $(mn,mk,m,n,m\lambda,\frac{m}{n}k^2)$. 
The property that $|W|\J_{m,n}=\J_{m,n}|W|=kJ_{mn}$ is easy to see. 
\end{proof}

\begin{example}
Let $H$ be a $\mathrm{GH}(2^n,2^{2nt-n})$ over the cyclic group $C_{2^n}$ generated by the nega-shift matrix of order $2^{n-1}$, and let $t$ and $n$ be positive integers. Let $W$ be a weighing matrix $W(2^{n-1},2^{n-1}-1)$. Then the matrix $[H_{ij}W]$ is a quasi-balanced weighing matrix whose absolute value matrix is the incidence matrix of a $$GDD(2^{2nt+n-1},2^{2nt}(2^{n-1}-1),2^{2nt},2^{n-1},2^{2nt}(2^{n-1}-2),2^{2nt-n+1}(2^{2n-2}-2^n+1))).$$
\end{example}

\begin{theorem}\label{thm:qbsgdd2}
If there exist a $\bgw(m+1,m,m-1)$ over the cyclic group of order $2n$ and a $\bgw(n,k,\lambda)$ over $\{1,-1\}$, then there exists a quasi-balanced weighing matrix $\w((m+1)n,mk)$, say $W$, for which $|W|$ is an
$$
\sgdd\left((m+1)n,m k,m+1,n,m\lambda,\frac{m-1}{n}k^2\right)
$$ 
and satisfies the property that $|W|\J_{m,n}=\J_{m,n}|W|=k(J_{mn}-I_{m}\otimes J_n)$.
\end{theorem}

\begin{proof}
Let $C_1$ be a $\bgw(m+1,m,m-1)$ over the cyclic group of order $2n$ generated by the negacirculant matrix $N$ of order $n$, and $C_2$ be a $\bgw(n,k,\lambda)$ over $\{1,-1\}$. For the matrix $C_1$, write the $(i,j)$-entry of $C_1$ as $(1-\delta_{ij})N^{c_{ij}}$ for $i,j\in\{1,\ldots,m+1\}$. Define an $(m+1) \times (m+1)$ $(0,\pm1)$ block matrix $W$ by $W_{ij}=(1-\delta_{ij})N^{c_{ij}}C_2$. Then $W$ is the desired quasi-balanced weighing matrix. The rest of the proof is similar to that of Theorem~\ref{thm:qbsgdd1}. 
\end{proof}

\begin{corollary}
 If there exists a $\bgw(n,k,\lambda; \{-1,1\})$, then there exists a quasi-balanced $\w((m+1)n,mk)$ where $\abs{W}$ is an
 $$
 \sgdd\left((m+1)n,m k,m+1,n,m\lambda,\frac{m-1}{n}k^2\right),
 $$
 for infintely many primes $m$.
\end{corollary}

\begin{proof}
 Note that $m$ would be of the form $2nk+1$, for some $k \in \N$. By Dirichlet's result on arithmetic progressions \cite[Corollary 4.10]{mult-num}, there are infinitely many primes $m$ of the required form.
\end{proof}

\begin{theorem}\label{thm:qbsgdd3}
If there exist a $\bgw(m+1,m,m-1)$ over the cyclic group of order $2n$ and a negacirculant $\bgw(n,k,\lambda)$ over $\{1,-1\}$, then there exists a quasi-balanced weighing matrix $\w((m+1)n,mk)$, say $W$, for which $|W|$ is a
$$
\ddg\left((m+1)n,m k,m+1,n,m\lambda,\frac{m-1}{n}k^2\right)
$$ 
and satisfies the property that $|W|\J_{m,n}=\J_{m,n}|W|=k(J_{mn}-I_{m}\otimes J_n)$.
\end{theorem}

\begin{proof}
Let $C_1$ be a $\bgw(m+1,m,m-1)$ over the cyclic group of order $2n$ generated by the negacirculant matrix $N$ of order $n$, and $C_2$ be a negacirculant $\bgw(n,k,\lambda)$ over $\{1,-1\}$. For the matrix $C_1$, write the $(i,j)$-entry of $C_1$ as $(1-\delta_{ij})N^{c_{ij}}$ for $i,j\in\{1,\ldots,m+1\}$. Define an $(m+1) \times (m+1)$ $(0,\pm1)$ block matrix $W$ by $W_{ij}=(1-\delta_{ij})N^{c_{ij}}C_2 R$. Recall that $R$ is the back diagonal matrix. Then $W$ is the desired quasi-balanced weighing matrix, and the proof is same as  Theorem~\ref{thm:qbsgdd2}. 
\end{proof}

\begin{corollary}
 If there exists a $\bgw(n,k,\lambda; \{-1,1\})$, then there exists a quasi-balanced $\w((m+1)n,mk)$ where $\abs{W}$ is a 
 $$
 \ddg((m+1)n,m k,m+1,n,m\lambda,\frac{m-1}{n}k^2),
 $$
 for infinitely many primes $m$.
\end{corollary}

\subsection{Computer searches}

To make a computer search feasible the case of a quasi-balanced signing of the adjacency matrix of a strongly regular graph is considered in this section.

Let $W$ be a quasi-balanced signing of an SRG$(v,k,\lambda,\mu)$ over $\R_n$, and consider the multisets
$$
S_{ij} = \{W_{i\ell}\overline{W_{j\ell}} : 0 \leq \ell < v, W_{i\ell} \neq 0, W_{j\ell} \neq 0\}, \qquad 0 \leq i < j < v.
$$
If each element of $\R_n$ appears either $\lambda/n$ or $\mu/n$ times in every $S_{ij}$ predicated upon whether or not $W_{ij} \neq 0$, then we say that the signing is \emph{srg-balanced}. Note that this assumes $n$ is a common divisor of $\lambda$ and $\mu$.

Assume now that $W$ corresponds to a balanced signing of an SRG$(v,k,\lambda,\mu)$ with adjacency matrix $A$, and continue to assume that $i \neq j$. If the $i$-th and $j$-th rows of $A$ are incident in $\lambda$ nonzero positions, i.e. $W_{ij} \neq 0$, then the conjugate inner product of the corresponding rows of $W$ must contain each element of $\R_n$ a constant number of times, namely, $\lambda/n$ times. Similarly, if $W_{ij} = 0$, then the product between the corresponding rows must contain each element of $\R_n$ $\mu/n$ times. Put another way, $W$ corresponds to an srg-balanced signing if and only if
$$
WW^* = kI_v + \frac{\lambda}{n}\left(\sum_{\alpha \in R_n}\alpha\right)A + \frac{\mu}{n}\left(\sum_{\alpha \in R_n}\alpha\right)(J_v - I_v - A).
$$
Since $\sum_{\alpha \in \R_n}\alpha = 0$, we obtain again that $WW^* = kI_v$. Evidently, then, an srg-balanced signing implies the property of being a Butson weighing matrix. The converse is not true in general, however.

\begin{example}\label{strictly-quasi-balanced}
 It was shown in \cite{gm3} that an SRG$(16,6,2,2)$ cannot have an srg-balanced signing; however, we can obtain a stricly quasi-balanced signing over $\{-1,1,-i,i\}$. See Figure \ref{srg-16-6-2-2} of the appendix.
\end{example}

In the case that $n=p$ is prime, we can then say more. The following can be found in \cite{lam-leung}.

\begin{lemma}
 Let $p$ be prime, and let $\omega$ be a primitive $p$-th root of unity. Then $\sum_{j=0}^n a_j \omega^j = 0$ for some $n < p$ and $a_0, \dots, a_n \in \N$ not all zero if and only if $n=p-1$ and $a_0 = \cdots = a_n$.
\end{lemma}

\begin{proof}
 Take $f(x)=\sum_{j=0}^n a_j x^j \in \mathbb{Z}[x]$. The $\Q$-minimal polynomial of $\omega$ is $h(x)=\sum_{j=0}^{p-1}x^j$. Since $f(\omega)=0$, $h \mid f$. Moreover, since $\text{deg }f < p$, it follows that $f = kh$, for some $k \in \Z$.
\end{proof}

We can now state the following.

\begin{theorem}
 Let $A$ be the adjacency matrix of an SRG$(v,k,\lambda,\mu)$ admitting a signing over the $p$-th roots of unity for some prime $p$. Then $p$ is a common divisor of $\lambda$ and $\mu$, and, moreover, the signing is srg-balanced.
\end{theorem}

\begin{corollary}
 The quasi-balanced weighing matrices constructed in Theorems \ref{thm:cons1}, \ref{thm:qbsgdd1}, \ref{thm:qbsgdd2}, and \ref{thm:qbsgdd3} are srg-balanced.
\end{corollary}

Using a table of the adjacency matrices of inequivalent, small parameter strongly regular graphs, we endeavoured to sign these matrices in order to obtain srg-balanced signings. The only non-trivial common divisors of the indices of the tabulated parameter sets are 2 and 3. The following are the results we obtained.

\begin{longtable}[c.]{cccc}
 \caption{Strongly regular graphs admitting an srg-balanced signing.$^*$}\\

\it Parameters & $\Z_2$ & $\Z_3$ & \it Reference \\
\hline
 \endfirsthead

 \multicolumn{4}{c}{\it Continuation of Table \ref{long}.}\\
 \it Parameters & $\Z_2$ & $\Z_3$ & \it Reference \\
 \hline
 \endhead


5-2-0-1	& --- & --- & --- \\ 
9-4-1-2 & --- & --- & --- \\
10-3-0-1 & --- & --- & --- \\ 
10-6-3-4 & --- & --- & --- \\
13-6-2-3 & --- & --- & --- \\
15-6-1-3 & --- & --- & --- \\ 
15-8-4-4 & NO & --- & \cite{gm3} \\
16-5-0-2 & YES$^\dag$ (Figure \ref{srg-16-5-0-2}) & --- & NEW \\ 
16-6-2-2 & NO$^\ddag$ & --- & \cite{sch} \\ 
16-9-4-6 & YES (Figure \ref{sgr-16-9-4-6})& --- & NEW \\
16-10-6-6 & NO & NO & \cite{gm3} \\
17-8-3-4 & --- & --- & --- \\ 
21-10-3-6 & --- & NO & NEW \\
21-10-5-4 & --- & --- & --- \\
25-8-3-2 & --- & --- & --- \\
25-12-5-6 & --- & --- & --- \\
25-16-9-12 & --- & ? & --- \\ 
26-10-3-4 & --- & --- & --- \\ 
26-15-8-9 & --- & --- & --- \\
27-10-1-5 & --- & --- & --- \\ 
27-16-10-8 & ? & --- & --- \\
28-12-6-4 & YES (Figure \ref{srg-28-12-6-4})& --- & NEW \\ 
28-15-6-10 & ? & --- & --- \\
29-14-6-7 & --- & --- & --- \\ 

\end{longtable}

$^*$The signings obtained are shown in the appendix.

$^\dag$Note that the signing of the SRG$(16,5,0,2)$ over $\Z_2$ is unique up to isomorphism. The remaining signings obtained have not been characterized in this way.
 
$^\ddag$From Example \ref{strictly-quasi-balanced}, we see that the SRG$(16,6,2,2)$ admits a strictly quasi-balanced signing.
 
The entries given by a --- indicates that an srg-balanced signing is not possible as the $\lambda$ and $\mu$ of the SRG share no nontrivial divisors. If one were to broaden their search and admit any quasi-balanced signing, then many of the entries would continue to fail admiting a signing. This can be seen by noting necessary conditions such as if the number of vertices is odd, then the degree must be square, and if the number of vertices is $2\pmod{4}$, then the degree must be the sum of two squares, etc. The interested reader may profitably consult \cite{DLF,OD} for details.

\section{Quasi-balanced weighing matrices and association schemes}

We consider quasi-balanced weighing matrices $W$ such that $|W|$ satisfies
\begin{enumerate*}
\item $|W|$ is the adjacency matrix of a strongly regular graph, or
\item $|W|$ is the incidence matrix of a symmetric group divisible design such that $|W|\J_{m,n}=\J_{m,n}|W|=\frac{k}{m}J_v$ or that $|W|\J_{m,n}=\J_{m,n}|W|=\frac{k}{m-1}(J_v-\J_{m,n})$. 
\end{enumerate*}

In this section, we show that association schemes are obtained from these quasi-balanced weighing matrices and characterize the quasi-balanced weighing matrices with these association schemes.  

\subsection{Association schemes from signed strongly regular graphs}

Let $W=W_1-W_2$ be a quasi-balanced weighing matrix whose absolute matrix is a strongly regular graph,  where $W_1$ and $W_2$ are disjoint $(0,1)$-matrices. Then $|W|=W_1+W_2$. For simplicity, we take $|W|=A$.

Let $P=\begin{pmatrix}0 & 1 \\ 1 & 0\end{pmatrix}$, and define the adjacency matrices as follows:
\begin{align}
A_i&=\begin{pmatrix}
P^{i-1}\otimes I_v&0\\
0& P^{i-1}\otimes I_v
\end{pmatrix}\quad (0\leq i \leq 1) ,\label{eq:a1srg}\\
A_2&=\begin{pmatrix}
J_2 \otimes A&0\\
0&J_2 \otimes A
\end{pmatrix},\label{eq:a2srg}\displaybreak[0]\\
A_{3}&=\begin{pmatrix}
J_2 \otimes (J_v-I_v-A)&0\\
0&J_2 \otimes (J_v-I_v-A)
\end{pmatrix},\label{eq:a3srg}\displaybreak[0]\\
A_{4}&=\begin{pmatrix}
0&J_2\otimes I_v\\
J_2\otimes I_v&0
\end{pmatrix},\label{eq:a4srg}\displaybreak[0]\\
A_{5}&=\begin{pmatrix}
0&I_2\otimes W_1+P\otimes W_2\\
I_2\otimes W_1^\top+P\otimes W_2^\top&0
\end{pmatrix},\nonumber\\
A_{6}&=\begin{pmatrix}
0&I_2\otimes W_2+P\otimes W_1\\
I_2\otimes W_2^\top+P\otimes W_1^\top&0
\end{pmatrix},\nonumber\\
A_{7}&=\begin{pmatrix}
0&J_2\otimes(J_v-I_v-A)\\
J_2\otimes(J_v-I_v-A)&0
\end{pmatrix}.\nonumber
\end{align}


\begin{theorem}\label{thm:assrg}
Let $W$ be a quasi-balanced weighing matrix such that $|W|$ is an $\srg(v,k,\lambda,\mu)$. Then $\{A_i\}_{i=0}^7$ is an association scheme where the eigenmatrices $P$ and $Q$ are given by 
\begin{align*}
P&=\left(
\begin{array}{cccccccc}
 1 & 1 & 2 k & -\frac{2 k (s+1) (t+1)}{k+s t} & 2 & k & k & -\frac{2 k (s+1) (t+1)}{k+s t} \\
 1 & 1 & 2 k & -\frac{2 k (s+1) (t+1)}{k+s t} & -2 & -k & -k & \frac{2 k (s+1) (t+1)}{k+s t} \\
 1 & 1 & 2 s & -2 (s+1) & 2 & s & s & -2 (s+1) \\
 1 & 1 & 2 s & -2 (s+1) & -2 & -s & -s & 2 (s+1) \\
 1 & 1 & 2 t & -2 (t+1) & 2 & t & t & -2 (t+1) \\
 1 & 1 & 2 t & -2 (t+1) & -2 & -t & -t & 2 (t+1) \\
 1 & -1 & 0 & 0 & 0 & \sqrt{k} & -\sqrt{k} & 0 \\
 1 & -1 & 0 & 0 & 0 & -\sqrt{k} & \sqrt{k} & 0 \\
\end{array}\right), \\
Q&=\left(
\arraycolsep=0.75pt
\begin{array}{cccccccc}
 1 & 1 & \frac{k (k-t) (t+1)}{(t-s) (k+s t)} & \frac{k (k-t) (t+1)}{(t-s) (k+s t)} & \frac{k (k-s) (s+1)}{(s-t) (k+s t)} & \frac{k (k-s) (s+1)}{(s-t) (k+s t)} & \frac{(k-s) (k-t)}{k+s t} & \frac{(k-s) (k-t)}{k+s t} \\
 1 & 1 & \frac{k (k-t) (t+1)}{(t-s) (k+s t)} & \frac{k (k-t) (t+1)}{(t-s) (k+s t)} & \frac{k (k-s) (s+1)}{(s-t) (k+s t)} & \frac{k (k-s) (s+1)}{(s-t) (k+s t)} & -\frac{(k-s) (k-t)}{k+s t} & -\frac{(k-s) (k-t)}{k+s t} \\
 1 & 1 & \frac{s (t+1) (t-k)}{(s-t) (k+s t)} & \frac{s (t+1) (t-k)}{(s-t) (k+s t)} & \frac{(k-s) (s+1) t}{(s-t) (k+s t)} & \frac{(k-s) (s+1) t}{(s-t) (k+s t)} & 0 & 0 \\
 1 & 1 & \frac{t-k}{s-t} & \frac{t-k}{s-t} & \frac{k-s}{s-t} & \frac{k-s}{s-t} & 0 & 0 \\
 1 & -1 & \frac{k (k-t) (t+1)}{(t-s) (k+s t)} & \frac{k (k-t) (t+1)}{(s-t) (k+s t)} & \frac{k (k-s) (s+1)}{(s-t) (k+s t)} & -\frac{k (k-s) (s+1)}{(s-t) (k+s t)} & 0 & 0 \\
 1 & -1 & \frac{s (t+1) (t-k)}{(s-t) (k+s t)} & \frac{s (k-t) (t+1)}{(s-t) (k+s t)} & \frac{(k-s) (s+1) t}{(s-t) (k+s t)} & -\frac{(k-s) (s+1) t}{(s-t) (k+s t)} & \frac{(k-s) (k-t)}{\sqrt{k} (k+s t)} & -\frac{(k-s) (k-t)}{\sqrt{k} (k+s t)} \\
 1 & -1 & \frac{s (t+1) (t-k)}{(s-t) (k+s t)} & \frac{s (k-t) (t+1)}{(s-t) (k+s t)} & \frac{(k-s) (s+1) t}{(s-t) (k+s t)} & -\frac{(k-s) (s+1) t}{(s-t) (k+s t)} & -\frac{(k-s) (k-t)}{\sqrt{k} (k+s t)} & \frac{(k-s) (k-t)}{\sqrt{k} (k+s t)} \\
 1 & -1 & \frac{t-k}{s-t} & \frac{k-t}{s-t} & \frac{k-s}{s-t} & \frac{s-k}{s-t} & 0 & 0 \\
\end{array}
\right), 
\end{align*}
where $k>s>t$ are the distinct eigenvalues of the strongly regular graph.  
\end{theorem}

\begin{proof}
Since $W=W_1-W_2$ is a weighing matrix, and since $W_1+W_2$ is the adjacency matrix of a strongly regular graph with parameters $(v,k,\lambda,\mu)$, we have 
\begin{align*}
W_1W_1^\top+W_2W_2^\top&=W_1^\top W_1+W_2^\top W_2=\frac{1}{2}(2k I_v+\lambda A+\mu(J_v-I_v-A)), \\
W_1W_2^\top+W_2W_1^\top&=W_1^\top W_2+W_2^\top W_1=\frac{1}{2}(\lambda A+\mu(J_v-I_v-A)).
\end{align*}
It follows readily from the equations above that the $A_i$'s form an association scheme. 

It is straightforward to see that the intersection array $B_4$ is given by 
$$B_5=
\begin{pmatrix}
 0 & 0 & 0 & 0 & 0 & 1 & 0 & 0 \\
 0 & 0 & 0 & 0 & 0 & 0 & 1 & 0 \\
 0 & 0 & 0 & 0 & k & \lambda & \lambda & \mu \\
 0 & 0 & 0 & 0 & 0 & k-\lambda-1 & k-\lambda-1 & k-\mu \\
 0 & 0 & 1 & 0 & 0 & 0 & 0 & 0 \\
 k & 0 & \frac{\lambda}{2} & \frac{\mu}{2} & 0 & 0 & 0 & 0 \\
 0 & k & \frac{\lambda}{2} & \frac{\mu}{2} & 0 & 0 & 0 & 0 \\
 0 & 0 & k-\lambda-1 & k-\mu & 0 & 0 & 0 & 0
\end{pmatrix}.
$$
By the formula 
$$
v=\frac{(k-s)(k-t)}{k+st},\quad\lambda=k+s+t+st,\quad \mu=k+st,
$$ 
apply \cite[Theorem 4.1]{BI} to this case to obtain the desired eigenmatrices. 
\end{proof}

\begin{theorem}
If there exists an association scheme with the eigenmatrices given in Theorem~\ref{thm:assrg}, then there exists quasi-balanced weighing matrix $W$ such that $|W|$ is an $\srg(v,k,\lambda,\mu)$ where $v=\frac{(k-s)(k-t)}{k+st},\lambda=k+s+t+st,\mu=k+st$.
\end{theorem}

\begin{proof}
Consider the matrix $A_1+A_2+A_3$. Its eigenvalues are $2v-1$ and $-1$ with multiplicities $2$ and $4v-2$, respectively. Then $A_1+A_2+A_3$ is the adjacency matrix of two copies of $K_{2v}$, the complete graph on $2v$ vertices. 

Next, by the eigenvalues of $A_1$, it is the adjacency matrix of a graph of disjoint $2v$ edges. We may assume that $A_1$ is of the form in \eqref{eq:a1srg}. 

Now, consider $A_4$. By the given eigenmatrices, we find that $A_1A_4=A_4A_1=A_4$, which yields 
$$
A_4=\begin{pmatrix}0 & J_2\otimes X \\ J_2\otimes X^\top & 0 \end{pmatrix}, 
$$
for some matrix $X$. 
Again by the eigenmatrices, we find $A_4^2=2A_0$ and thus $XX^\top=I$. 
Therefore, $X$ is a permutation matrix and suitably rearranging the vertices yields that $A_4$ is of the desired form in \eqref{eq:a4srg}.

Consider $A_2$. 
By the given eigenmatrices, we find that $A_1A_2=A_2A_1=A_2$ and $A_2A_4=A_4A_2$, which yields 
$$
A_2=\begin{pmatrix}J_2\otimes A & 0 \\ 0 & J_2\otimes A \end{pmatrix},  
$$
for some matrix $A$. 
Therefore, we have the form in \eqref{eq:a3srg}. 
Again by the eigenmatrices, we find $A_2^2=2(kA_0+kA_1+\lambda A_2+\mu A_3)$ and thus $A^2=kI+\lambda A +\mu(J-I-A)$. 
This shows that $A$ is the adjacency matrix of a strongly regular graph with parameters $(v,k,\lambda,\mu)$. 
  
By the given eigenmartrices, we find that $A_1A_4=A_5$, which yields 
$$A_5=\begin{pmatrix}  0 & 0 & W_1 & W_2 \\
 0 & 0 & W_2 & W_1 \\
 W_1^\top & W_2^\top & 0 & 0 \\
 W_2^\top & W_1^\top & 0 & 0
\end{pmatrix},
A_6=\begin{pmatrix} 0 & 0 & W_2 & W_1 \\
 0 & 0 & W_1 & W_2 \\
 W_2^\top & W_1^\top & 0 & 0 \\
 W_1^\top & W_2^\top & 0 & 0 
\end{pmatrix},$$ 
for some $(0,1)$-matrices $W_1$ and $W_2$. 
Again, by the eigenmatrices, we have that $A_4A_5=A_2$, that is, $W_1+W_2=A$, and that  
\begin{align*}
A_5^2&=kA_0+\frac{1}{2}(k+s+t+st)A_2+\frac{1}{2}(k+st)A_3, \\
A_5A_6&=A_6A_5=kA_1+\frac{1}{2}(k+s+t+st)A_2+\frac{1}{2}(k+st)A_3, \\
A_6^2&=kA_0+\frac{1}{2}(k+s+t+st)A_2+\frac{1}{2}(k+st)A_3.
\end{align*}
From these identities,  it follows that $(A_5-A_6)^2=2k(A_0-A_1)$, that is,  
\begin{align*}
(W_1-W_2)(W_1^\top-W_2^\top)&=kI. 
\end{align*}
Hence, $W_1-W_2$ is a quasi-balanced weighing matrix whose absolute matrix is a strongly regular graph  with the desired parameters.  
\end{proof}

\subsection{Association schemes from signed symmetric group divisible designs}

The constructed quasi-balanced weighing matrices $W$ in Theorems~\ref{thm:qbsgdd1} and \ref{thm:qbsgdd2} have the property that $|W|\J_{m,n}=\J_{m,n}|W|=\frac{k}{m}J_v$ or that $|W|\J_{m,n}=\J_{m,n}|W|=\frac{k}{m-1}(J_v-\J_{m,n})$.
In this subsection, we provide a relation between such quasi-balanced weighing matrices and association schemes.  

\subsubsection{The case that $|W|\J_{m,n}=\J_{m,n}|W|=\frac{k}{m}J_v$}

For a weighing matrix $W$, write $W=W_1-W_2$ where $W_1$ and $W_2$ are disjoint $(0,1)$-matrices. 
Then it readily follows that $W$ is a quasi-balanced whose absolute matrix is a symmetric group divisible design if and only if $W_1+W_2$ is the incidence matrix of a symmetric group divisible design with parameters $(v,k,m,n,\lambda_1,\lambda_2)$. 

Let $P=\begin{pmatrix}0 & 1 \\ 1 & 0\end{pmatrix}$. 
Define the adjacency matrices as follows:
\begin{align}
A_i&=\begin{pmatrix}
P^{i-1}\otimes I_v&0\\
0& P^{i-1}\otimes I_v
\end{pmatrix}\quad (0\leq i \leq 1) ,\label{eq:a1gdd1}\displaybreak[0]\\ 
A_2&=\begin{pmatrix}
J_2 \otimes (\J_{m,n}-I_v)&0\\
0&J_2 \otimes (\J_{m,n}-I_v)
\end{pmatrix},\label{eq:a2gdd1}\displaybreak[0]\\
A_{3}&=\begin{pmatrix}
J_2 \otimes (J_v-\J_{m,n})&0\\
0&J_2 \otimes (J_v-\J_{m,n})
\end{pmatrix},\label{eq:a3gdd1}\displaybreak[0]\\
A_{4}&=\begin{pmatrix}
0&I_2\otimes W_1+P\otimes W_2\\
I_2\otimes W_1^\top+P\otimes W_2^\top&0
\end{pmatrix},\nonumber\displaybreak[0]\\
A_{5}&=\begin{pmatrix}
0&I_2\otimes W_2+P\otimes W_1\\
I_2\otimes W_2^\top+P\otimes W_1^\top&0
\end{pmatrix},\nonumber\displaybreak[0]\\
A_{6}&=\begin{pmatrix}
0&J_2\otimes(J_v-W_1-W_2)\\
J_2\otimes(J_v-W_1^\top-W_2^\top)&0
\end{pmatrix}.\nonumber
\end{align}


\begin{theorem}\label{thm:asgdd1}
Let $W$ be a quasi-balanced weighing matrix such that $|W|$ is an $\sgdd(v,k,m,n,\lambda_1,\lambda_2)$ with the property that $|W|\J_{m,n}=\J_{m,n}|W|=\frac{k}{m}J_v$.
If $k<v$, then $\{A_i\}_{i=0}^6$ is an association scheme with the eigenmatrices $P$ and $Q$ given by 
\begin{align*}
P&=
\left(
\begin{array}{ccccccc}
 1 & 1 & 2 (n-1) & 2 (m-1) n & k & k & 2 (m n- k) \\
 1 & -1 & 0 & 0 & \sqrt{k} & -\sqrt{k} & 0 \\
 1 & 1 & 2 (n-1) & -2 n & 0 & 0 & 0 \\
 1 & -1 & 0 & 0 & -\sqrt{k} & \sqrt{k} & 0 \\
 1 & 1 & 2 (n-1) & 2 (m-1) n & -k & -k & 2 (k-m n) \\
 1 & 1 & -2 & 0 & \frac{\sqrt{k (m n-k)}}{\sqrt{m(n-1)}} & \frac{\sqrt{k (m n-k)}}{\sqrt{m(n-1)}} & -\frac{2 \sqrt{k (m n-k)}}{\sqrt{m(n-1)}} \\
 1 & 1 & -2 & 0 & -\frac{\sqrt{k (m n-k)}}{\sqrt{m(n-1)}} & -\frac{\sqrt{k (m n-k)}}{\sqrt{m(n-1)}} & \frac{2 \sqrt{k (m n-k)}}{\sqrt{m(n-1)}} \\
\end{array}
\right),\displaybreak[0]\\
Q&=\left(
\begin{array}{ccccccc}
 1 & m n & 2 (m-1) & m n & 1 & m (n-1) & m (n-1) \\
 1 & -m n & 2 (m-1) & -m n & 1 & m (n-1) & m (n-1) \\
 1 & 0 & 2 (m-1) & 0 & 1 & -m & -m \\
 1 & 0 & -2 & 0 & 1 & 0 & 0 \\
 1 & \frac{m n}{\sqrt{k}} & 0 & -\frac{m n}{\sqrt{k}} & -1 & \frac{\sqrt{m(n-1)(m n-k)}}{\sqrt{k}} & -\frac{\sqrt{m(n-1)(m n-k)}}{\sqrt{k}} \\
 1 & -\frac{m n}{\sqrt{k}} & 0 & \frac{m n}{\sqrt{k}} & -1 & \frac{\sqrt{m(n-1)(m n-k)}}{\sqrt{k}} & -\frac{\sqrt{m(n-1)(m n-k)}}{\sqrt{k}} \\
 1 & 0 & 0 & 0 & -1 & -\frac{\sqrt{km(n-1)}}{\sqrt{m n-k}} & \frac{\sqrt{km(n-1)}}{\sqrt{m n-k}} \\
\end{array}
\right).
\end{align*}
\end{theorem}

\begin{proof}
Since $W=W_1-W_2$ is a weighing matrix and $W_1+W_2$ is the incidence matrix of a symmetric group divisible design with parameters $(v,k,m,n,\lambda_1,\lambda_2)$, we have 
\begin{align*}
W_1W_1^\top+W_2W_2^\top&=W_1^\top W_1+W_2^\top W_2=\frac{1}{2}((2k-\lambda_1)I_v+(\lambda_1-\lambda_2)\J_{m,n}+\lambda_2 J_v), \\
W_1W_2^\top+W_2W_1^\top&=W_1^\top W_2+W_2^\top W_1=\frac{1}{2}(-\lambda_1I_v+(\lambda_1-\lambda_2)\J_{m,n}+\lambda_2 J_v),\\
(W_1+W_2)\J_{m,n}&=\J_{m,n}(W_1+W_2)=\frac{k}{m}J_v. 
\end{align*}
It follows readily from the equations above that the $A_i$'s form an association scheme. 

It is straightforward to see that the intersection matrix $B_4$ is given by 
$$B_4=
\begin{pmatrix}
 0 & 0 & 0 & 0 & 1 & 0 & 0 \\
 0 & 0 & 0 & 0 & 0 & 1 & 0 \\
 0 & 0 & 0 & 0 & \frac{k}{m}-1 & \frac{k}{m}-1 & \frac{k}{m} \\
 0 & 0 & 0 & 0 & k-\frac{k}{m}  & k-\frac{k}{m} & k-\frac{k}{m}  \\
 k & 0 & \frac{\lambda_1}{2} & \frac{\lambda_2}{2} & 0 & 0 & 0 \\
 0 & k & \frac{\lambda_1}{2} & \frac{\lambda_2}{2} & 0 & 0 & 0 \\
 0 & 0 & k-\lambda_1 &  k-\lambda_2 & 0 & 0 & 0  \\
\end{pmatrix},$$
Apply \cite[Theorem 4.1]{BI} to this case to obtain the desired eigenmatrices. 
\end{proof}

\begin{theorem}\label{thm:asgddcar1}
Assume $mn>k$. 
If there exists an association scheme with the eigenmatrices given in Theorem~\ref{thm:asgdd1}, then there exists a quasi-balanced weighing matrix $W$ such that $|W|$ is a $\sgdd(v,k,m,n,\lambda_1,\lambda_2)$ with the property that $|W|\J_{m,n}=\J_{m,n}|W|=\frac{k}{m}J_v$.
\end{theorem}

\begin{proof}
Consider the matrix $A_1+A_2+A_3$. Its eigenvalues are $2mn-1$ and $-1$ with multiplicities $2$ and $4mn-2$ respectively. Then $A_1+A_2+A_3$ is the adjacency matrix of two copies of $K_{2mn}$, the complete graph on $2mn$ vertices. 
Next, by the eigenvalues of $A_1$, it is the adjacency matrix of a graph of disjoint $2v$ edges.  
We may assume that $A_1$ is of the form in \eqref{eq:a1gdd1}.  
Thirdly consider $A_0+A_1+A_2$. 
By the given eigenmatrices, we find that $A_1(A_0+A_1+A_2)=(A_0+A_1+A_2)A_1=A_0+A_1+A_2$, which yields 
$$
A_0+A_1+A_2=\begin{pmatrix}J_2\otimes X & 0 \\ 0 & J_2\otimes Y \end{pmatrix},  
$$
for some matrices $X$ and $Y$. 
Again by the eigenmatrices, we find $(A_0+A_1+A_2)^2=2n(A_0+A_1+A_2)$ and thus $X^2=nX$ and $Y^2=nY$. This equation shows that $X=Y=\J_{m,n}$ after suitably permuting the columns and rows. 
Therefore we have the forms in \eqref{eq:a2gdd1} and \eqref{eq:a3gdd1}. 
  
By the given eigenmatrices, we find that $A_1A_4=A_5$, which yields 
$$A_4=\begin{pmatrix}  0 & 0 & W_1 & W_2 \\
 0 & 0 & W_2 & W_1 \\
 W_1^\top & W_2^\top & 0 & 0 \\
 W_2^\top & W_1^\top & 0 & 0
\end{pmatrix},
A_5=\begin{pmatrix} 0 & 0 & W_2 & W_1 \\
 0 & 0 & W_1 & W_2 \\
 W_2^\top & W_1^\top & 0 & 0 \\
 W_1^\top & W_2^\top & 0 & 0 
\end{pmatrix},$$ 
for some $(0,1)$-matrices $W_1$ and $W_2$. 
Again, by the eigenmatrices, we have that 
\begin{align*}
A_4^2&=kA_0+\frac{(k-m) k}{2m (n-1)}A_2+\frac{k^2}{2m n}A_3, \\
A_4A_5&=A_5A_4=kA_1+\frac{(k-m) k}{2m (n-1)}A_2+\frac{k^2}{2m n}A_3, \\
A_5^2&=kA_0+\frac{(k-m) k}{2m (n-1)}A_2+\frac{k^2}{2m n}A_3.
\end{align*}
From these identities, it follows that $(A_4-A_5)^2=2k(A_0-A_1), (A_4+A_5)^2=2k(A_0+A_1)+\frac{(k-m) k}{m (n-1)}A_2+\frac{k^2}{m n}A_3$, that is,  
\begin{align*}
(W_1-W_2)(W_1^\top-W_2^\top)&=kI_{mn},\\ 
(W_1+W_2)(W_1^\top+W_2^\top)=(W_1^\top+W_2^\top)(W_1+W_2)&=(k-\lambda)I_v+(\lambda_1-\lambda_2)\J_{m,n}+\lambda_2 J_{mn}. 
\end{align*}
Finally, 
\begin{align*}
(A_0+A_2)A_4=A_4(A_0+A_2)=\frac{k}{m}(A_4+A_5+A_6)
\end{align*}
implies that 
$$
(W_1+W_2)\J_{m,n}=\J_{m,n}(W_1+W_2)=\frac{k}{m}J_v. 
$$
Hence, $W_1-W_2$ is a quasi-balanced weighing matrix with the desired properties and parameters.  
\end{proof}

\subsubsection{The case that $|W|\J_{m,n}=\J_{m,n}|W|=\frac{k}{m-1}(J_v-\J_{m,n})$}

Let $W=W_1-W_2$ be a quasi-balanced weighing matrix with the property that $(W_1+W_2)\J_{m,n}=\J_{m,n}(W_1+W_2)=\frac{k}{m-1}(J_v-\J_{m,n})$. 

Let $P=\begin{pmatrix}0 & 1 \\ 1 & 0\end{pmatrix}$. 
Define the adjacency matrices as follows:
\begin{align}
A_i&=\begin{pmatrix}
P^{i-1}\otimes I_v&0\\
0& P^{i-1}\otimes I_v
\end{pmatrix}\quad (0\leq i \leq 1) ,\label{eq:a1gdd2}\\
A_2&=\begin{pmatrix}
J_2 \otimes (\J_{m,n}-I_v)&0\\
0&J_2 \otimes (\J_{m,n}-I_v)
\end{pmatrix},\label{eq:a2gdd2}\\
A_{3}&=\begin{pmatrix}
J_2 \otimes (J_v-\J_{m,n})&0\\
0&J_2 \otimes (J_v-\J_{m,n})
\end{pmatrix},\label{eq:a3gdd2}\\
A_{4}&=\begin{pmatrix}
0&I_2\otimes W_1+P\otimes W_2\\
I_2\otimes W_1^\top+P\otimes W_2^\top&0
\end{pmatrix},\nonumber\\
A_{5}&=\begin{pmatrix}
0&I_2\otimes W_2+P\otimes W_1\\
I_2\otimes W_2^\top+P\otimes W_1^\top&0
\end{pmatrix},\nonumber\\
A_{6}&=\begin{pmatrix}
0&J_2\otimes(J_v-W_1-W_2-\J_{m,n})\\
J_2\otimes(J_v-W_1^\top-W_2^\top-\J_{m,n})&0
\end{pmatrix},\nonumber\\
A_{7}&=\begin{pmatrix}
0&J_2\otimes \J_{m,n}\\
J_2\otimes \J_{m,n}&0
\end{pmatrix}. \label{eq:a3gdd7}
\end{align}


\begin{theorem}\label{thm:asgdd2}
Let $W$ be a quasi-balanced weighing matrix such that $|W|$ is an $\sgdd(v,k,m,n,\lambda_1,\lambda_2)$ with the property that $|W|\J_{m,n}=\J_{m,n}|W|=\frac{k}{m-1}(J_v-\J_{m,n})$.
If $k<v$, then $\{A_i\}_{i=0}^7$ is an association scheme with the eigenmatrices $P$ and $Q$ given by 
\begin{align*}
P&=
\arraycolsep=0.75pt
\left(
\begin{array}{cccccccc}
 1 & 1 & 2 (n-1) & 2 (m-1) n & k & k & 2 (-k+m n-n) & 2 n \\
 1 & -1 & 0 & 0 & \sqrt{k} & -\sqrt{k} & 0 & 0 \\
 1 & 1 & 2 (n-1) & -2 n & \frac{k}{1-m} & \frac{k}{1-m} & \frac{2 (k-m n+n)}{m-1} & 2 n \\
 1 & 1 & 2 (n-1) & -2 n & \frac{k}{m-1} & \frac{k}{m-1} & -\frac{2 (k-m n+n)}{m-1} & -2 n \\
 1 & -1 & 0 & 0 & -\sqrt{k} & \sqrt{k} & 0 & 0 \\
 1 & 1 & 2 (n-1) & 2 (m-1) n & -k & -k & 2 (k-m n+n) & -2 n \\
 1 & 1 & -2 & 0 & -\frac{\sqrt{k(m n-n-k)}}{ \sqrt{(m-1)(n-1)}} & -\frac{\sqrt{k(m n-n-k)}}{ \sqrt{(m-1)(n-1)}} & \frac{2\sqrt{k(m n-n-k)}}{ \sqrt{(m-1)(n-1)}} & 0 \\
 1 & 1 & -2 & 0 & \frac{\sqrt{k(m n-n-k)}}{ \sqrt{(m-1)(n-1)}} & \frac{\sqrt{k(m n-n-k)}}{ \sqrt{(m-1)(n-1)}} & -\frac{2\sqrt{k(m n-n-k)}}{ \sqrt{(m-1)(n-1)}} & 0 \\
\end{array}
\right), \\
Q&=
\arraycolsep=2.0pt
\left(
\begin{array}{cccccccc}
 1 & m n & m-1 & m-1 & m n & 1 & m (n-1) & m (n-1) \\
 1 & -m n & m-1 & m-1 & -m n & 1 & m (n-1) & m (n-1) \\
 1 & 0 & m-1 & m-1 & 0 & 1 & -m & -m \\
 1 & 0 & -1 & -1 & 0 & 1 & 0 & 0 \\
 1 & \frac{m n}{\sqrt{k}} & -1 & 1 & -\frac{m n}{\sqrt{k}} & -1 & -\frac{m \sqrt{(n-1)(mn-n-k)}}{\sqrt{k(m-1)}} & \frac{m \sqrt{(n-1)(mn-n-k)}}{\sqrt{k(m-1)}} \\
 1 & -\frac{m n}{\sqrt{k}} & -1 & 1 & \frac{m n}{\sqrt{k}} & -1 & -\frac{m \sqrt{(n-1)(mn-n-k)}}{\sqrt{k(m-1)}} & \frac{m \sqrt{(n-1)(mn-n-k)}}{\sqrt{k(m-1)}} \\
 1 & 0 & -1 & 1 & 0 & -1 & \frac{m \sqrt{k(n-1)}}{ \sqrt{(m-1)(mn-n-k)}} & -\frac{m \sqrt{k(n-1)}}{ \sqrt{(m-1)(mn-n-k)}} \\
 1 & 0 & m-1 & -m+1 & 0 & -1 & 0 & 0 \\
\end{array}
\right).
\end{align*}
\end{theorem}

\begin{proof}
Since $W=W_1-W_2$ is a weighing matrix and $W_1+W_2$ is the incidence matrix of a symmetric group divisible design with parameters $(v,k,m,n,\lambda_1,\lambda_2)$, we have 
\begin{align*}
W_1W_1^\top+W_2W_2^\top&=W_1^\top W_1+W_2^\top W_2=\frac{1}{2}((2k-\lambda_1)I_v+(\lambda_1-\lambda_2)\J_{m,n}+\lambda_2 J_v), \\
W_1W_2^\top+W_2W_1^\top&=W_1^\top W_2+W_2^\top W_1=\frac{1}{2}(-\lambda_1I_v+(\lambda_1-\lambda_2)\J_{m,n}+\lambda_2 J_v),\\
(W_1+W_2)\J_{m,n}&=\J_{m,n}(W_1+W_2)=\frac{k}{m-1}(J_v-\J_{m,n}).
\end{align*}
It follows readily from the equations above that the $A_i$'s form an association scheme. 

It is straightforward to see that the intersection matrix $B_4$ is given by 
$$B_4=
\begin{pmatrix}
 0 & 0 & 0 & 0 & 1 & 0 & 0 & 0 \\
 0 & 0 & 0 & 0 & 0 & 1 & 0 & 0 \\
 0 & 0 & 0 & 0 & \frac{k}{m-1}-1 & \frac{k}{m-1}-1 &  \frac{k}{m-1} & 0 \\
 0 & 0 & 0 & 0 & k-\frac{k}{m-1}  & k-\frac{k}{m-1}  & k-\frac{k}{m-1} & k \\
 k & 0 & \frac{\lambda_1}{2} & \frac{\lambda_2}{2} & 0 & 0  & 0 & 0 \\
 0 & k & \frac{\lambda_1}{2} & \frac{\lambda_2}{2} & 0 & 0  & 0 & 0 \\
 0 & 0 & k-\lambda_1 & k-\lambda_2-\frac{k}{m-1} & 0 & 0  & 0 & 0 \\ 
 0 & 0 & 0 & \frac{k}{m-1} & 0 & 0 & 0 & 0 
\end{pmatrix},$$
Apply \cite[Theorem 4.1]{BI} to this case to obtain the desired eigenmatrices. 
\end{proof}

\begin{theorem}
Assume $mn>k$. 
If there exists an association scheme with the eigenmatrices given in Theorem~\ref{thm:asgdd2}, then there exists a quasi-balanced weighing matrix $W$ with the property that $|W|$ is a symmetric group divisible design with parameters $(v,k,m,n,\lambda_1,\lambda_2)$ and $|W|\J_{m,n}=\J_{m,n}|W|=\frac{k}{m-1}(J_v-\J_{m,n})$.
\end{theorem}

\begin{proof}
In a similar manner, it is shown that  we have the forms in  \eqref{eq:a1gdd2},  \eqref{eq:a2gdd2}, and \eqref{eq:a3gdd2}. 
  
By the given eigenmatrices, we find that $A_1A_7=A_7$, which yields 
$$A_7=\begin{pmatrix}  0 & 0 & X & X \\
 0 & 0 & X & X \\
 X & X & 0 & 0 \\
 X & X & 0 & 0
\end{pmatrix}=(J_2-I_2)\otimes J_2 \otimes X,$$ 
for some symmetric $(0,1)$-matrix $X$. 
By the eigenvalues of $A_7$, the matrix $X$ has the eigenvalues $n$ and $0$ with multiplicities $m$ and $mn-m$, respectively. By suitably rearranging the vertices, we have  the form in \eqref{eq:a3gdd7}. 
The rest is same as proof of Theorem~\ref{thm:asgddcar1}.
\end{proof}

\section*{Acknowledgments.}
Hadi Kharaghani is supported by the Natural Sciences and
Engineering  Research Council of Canada (NSERC).  Sho Suda is supported by JSPS KAKENHI Grant Number 18K03395.
\begin{rezabib}
\bib{arasu}{article}{
   author={Arasu, K. T.},
   author={Dillon, J. F.},
   author={Leung, Ka Hin},
   author={Ma, Siu Lun},
   title={Cyclic relative difference sets with classical parameters},
   journal={J. Combin. Theory Ser. A},
   volume={94},
   date={2001},
   number={1},
   pages={118--126},
   issn={0097-3165},
   review={\MR{1816250}},
   doi={10.1006/jcta.2000.3137},
}

\bib{BI}{book}{
   author={Bannai, Eiichi},
   author={Ito, Tatsuro},
   title={Algebraic combinatorics. I},
   note={Association schemes},
   publisher={The Benjamin/Cummings Publishing Co., Inc., Menlo Park, CA},
   date={1984},
   pages={xxiv+425},
   isbn={0-8053-0490-8},
   review={\MR{882540}},
}

\bib{BJL}{book}{
   author={Beth, Thomas},
   author={Jungnickel, Dieter},
   author={Lenz, Hanfried},
   title={Design theory. Vols. I,II},
   series={Encyclopedia of Mathematics and its Applications},
   volume={69},
   edition={2},
   publisher={Cambridge University Press, Cambridge},
   date={1999},
   pages={xx+1100},
   isbn={0-521-44432-2},
   review={\MR{1729456}},
   doi={10.1017/CBO9780511549533},
}

\bib{BCN}{book}{
   author={Brouwer, A. E.},
   author={Cohen, A. M.},
   author={Neumaier, A.},
   title={Distance-regular graphs},
   series={Ergebnisse der Mathematik und ihrer Grenzgebiete (3) [Results in
   Mathematics and Related Areas (3)]},
   volume={18},
   publisher={Springer-Verlag, Berlin},
   date={1989},
   pages={xviii+495},
   isbn={3-540-50619-5},
   review={\MR{1002568}},
   doi={10.1007/978-3-642-74341-2},
}

\bib{BVM}{book}{
   author={Brouwer, Andries E.},
   author={Van Maldeghem, H.},
   title={Strongly regular graphs},
   series={Encyclopedia of Mathematics and its Applications},
   volume={182},
   publisher={Cambridge University Press, Cambridge},
   date={2022},
   pages={xvii+462},
   isbn={978-1-316-51203-6},
   review={\MR{4350112}},
}

\bib{CVL}{book}{
   author={Cameron, P. J.},
   author={van Lint, J. H.},
   title={Designs, graphs, codes and their links},
   series={London Mathematical Society Student Texts},
   volume={22},
   publisher={Cambridge University Press, Cambridge},
   date={1991},
   pages={x+240},
   isbn={0-521-41325-7},
   isbn={0-521-42385-6},
   review={\MR{1148891}},
   doi={10.1017/CBO9780511623714},
}

\bib{DLF}{book}{
   author={de Launey, Warwick},
   author={Flannery, Dane},
   title={Algebraic design theory},
   series={Mathematical Surveys and Monographs},
   volume={175},
   publisher={American Mathematical Society, Providence, RI},
   date={2011},
   pages={xii+298},
   isbn={978-0-8218-4496-0},
   review={\MR{2815992}},
   doi={10.1090/surv/175},
}

\bib{gm1}{article}{
   author={Gibbons, P. B.},
   author={Mathon, R.},
   title={Signings of group divisible designs and projective planes},
   journal={Australas. J. Combin.},
   volume={11},
   date={1995},
   pages={79--104},
   issn={1034-4942},
   review={\MR{1327323}},
}

\bib{gm2}{article}{
   author={Gibbons, Peter B.},
   author={Mathon, Rudolf},
   title={Construction methods for Bhaskar Rao and related designs},
   journal={J. Austral. Math. Soc. Ser. A},
   volume={42},
   date={1987},
   number={1},
   pages={5--30},
   issn={0263-6115},
   review={\MR{862718}},
}
 
 \bib{gm3}{article}{
   author={Gibbons, Peter B.},
   author={Mathon, Rudolf A.},
   title={Group signings of symmetric balanced incomplete block designs},
   booktitle={Proceedings of the Singapore conference on combinatorial
   mathematics and computing (Singapore, 1986)},
   journal={Ars Combin.},
   volume={23},
   date={1987},
   number={A},
   pages={123--134},
   issn={0381-7032},
   review={\MR{890133}},
}
 
 \bib{IS}{book}{
   author={Ionin, Yury J.},
   author={Shrikhande, Mohan S.},
   title={Combinatorics of symmetric designs},
   series={New Mathematical Monographs},
   volume={5},
   publisher={Cambridge University Press, Cambridge},
   date={2006},
   pages={xiv+520},
   isbn={978-0-521-81833-9},
   isbn={0-521-81833-8},
   review={\MR{2234039}},
   doi={10.1017/CBO9780511542992},
}

\bib{JK}{article}{
   author={Jungnickel, Dieter},
   author={Kharaghani, H.},
   title={Balanced generalized weighing matrices and their applications},
   journal={Matematiche (Catania)},
   volume={59},
   date={2004},
   number={1-2},
   pages={225--261 (2006)},
   issn={0373-3505},
   review={\MR{2243033}},
}
 
\bib{jt}{article}{
   author={Jungnickel, Dieter},
   author={Tonchev, Vladimir D.},
   title={Perfect codes and balanced generalized weighing matrices. II},
   journal={Finite Fields Appl.},
   volume={8},
   date={2002},
   number={2},
   pages={155--165},
   issn={1071-5797},
   review={\MR{1894509}},
   doi={10.1006/ffta.2001.0327},
}

\bib{jt-ii}{article}{
   author={Jungnickel, Dieter},
   author={Tonchev, Vladimir D.},
   title={Perfect codes and balanced generalized weighing matrices. II},
   conference={
      title={International Workshop on Coding and Cryptography},
      address={Paris},
      date={2001},
   },
   book={
      series={Electron. Notes Discrete Math.},
      volume={6},
      publisher={Elsevier Sci. B. V., Amsterdam},
   },
   date={2001},
   pages={10},
   review={\MR{1985233}},
}

 \bib{kt}{article}{
   author={Kharaghani, H.},
   author={Torabi, R.},
   title={On a decomposition of complete graphs},
   journal={Graphs Combin.},
   volume={19},
   date={2003},
   number={4},
   pages={519--526},
   issn={0911-0119},
   review={\MR{2031006}},
   doi={10.1007/s00373-003-0527-y},
}
 
\bib{kps21}{article}{
    author={Kharaghani, Hadi},
    author={Pender, Thomas},
    author={Suda, Sho},
    title={A family of balanced generalized weighing matrices, },
    journal={to appear in {\it Combinatorica}},
}

\bib{krw}{article}{
   author={Klin, Mikhail},
   author={Reichard, Sven},
   author={Woldar, Andrew},
   title={Siamese objects, and their relation to color graphs, association
   schemes and Steiner designs},
   journal={Bull. Belg. Math. Soc. Simon Stevin},
   volume={12},
   date={2005},
   number={5},
   pages={845--857},
   issn={1370-1444},
   review={\MR{2241348}},
}

\bib{lam-leung}{article}{
   author={Lam, T. Y.},
   author={Leung, K. H.},
   title={On vanishing sums of roots of unity},
   journal={J. Algebra},
   volume={224},
   date={2000},
   number={1},
   pages={91--109},
   issn={0021-8693},
   review={\MR{1736695}},
   doi={10.1006/jabr.1999.8089},
}

\bib{mult-num}{book}{
   author={Montgomery, Hugh L.},
   author={Vaughan, Robert C.},
   title={Multiplicative number theory. I. Classical theory},
   series={Cambridge Studies in Advanced Mathematics},
   volume={97},
   publisher={Cambridge University Press, Cambridge},
   date={2007},
   pages={xviii+552},
   isbn={978-0-521-84903-6},
   isbn={0-521-84903-9},
   review={\MR{2378655}},
}

 \bib{sch}{article}{
   author={Schellenberg, Paul J.},
   title={A computer construction for balanced orthogonal matrices},
   conference={
      title={Proceedings of the Sixth Southeastern Conference on
      Combinatorics, Graph Theory and Computing (Florida Atlantic Univ.,
      Boca Raton, Fla., 1975)},
   },
   book={
      publisher={Utilitas Math., Winnipeg, Man.},
   },
   date={1975},
   pages={513--522. Congressus Numerantium, No. XIV},
   review={\MR{0392606}},
}

\bib{OD}{article}{
   author={Seberry, Jennifer},
   title={Hadamard matrices, orthogonal designs and
   Clifford--Gastineau-Hills algebras},
   journal={Australas. J. Combin.},
   volume={71},
   date={2018},
   pages={452--467},
   issn={1034-4942},
   review={\MR{3801276}},
}

\bib{Tonchev}{book}{
   author={Tonchev, Vladimir D.},
   title={Combinatorial configurations: designs, codes, graphs},
   series={Pitman Monographs and Surveys in Pure and Applied Mathematics},
   volume={40},
   note={Translated from the Bulgarian by Robert A. Melter},
   publisher={Longman Scientific \& Technical, Harlow; John Wiley \& Sons,
   Inc., New York},
   date={1988},
   pages={xii+189},
   isbn={0-582-99483-7},
   review={\MR{940701}},
}
\end{rezabib}

\section*{Appendix}

\begin{figure}[H]
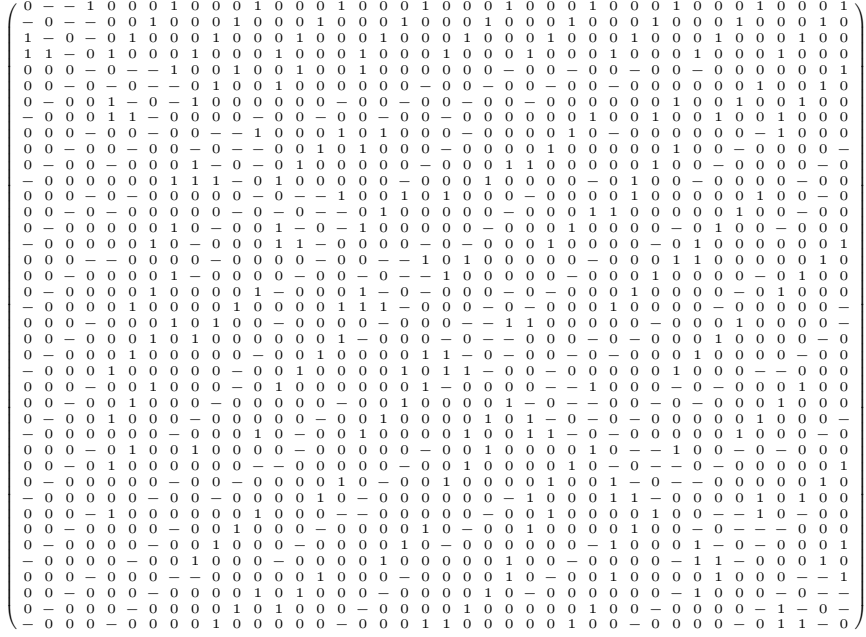

\begin{tiny}
 \[
 \arraycolsep=1.25pt\def\arraystretch{1}
  \left(\begin{array}{cccccccccccccccccccccccccccccccccccccccc}
0&-&-&1&0&0&0&1&0&0&0&1&0&0&0&1&0&0&0&1&0&0&0&1&0&0&0&1&0&0&0&1&0&0&0&1&0&0&0&1\\
-&0&-&-&0&0&1&0&0&0&1&0&0&0&1&0&0&0&1&0&0&0&1&0&0&0&1&0&0&0&1&0&0&0&1&0&0&0&1&0\\
1&-&0&-&0&1&0&0&0&1&0&0&0&1&0&0&0&1&0&0&0&1&0&0&0&1&0&0&0&1&0&0&0&1&0&0&0&1&0&0\\
1&1&-&0&1&0&0&0&1&0&0&0&1&0&0&0&1&0&0&0&1&0&0&0&1&0&0&0&1&0&0&0&1&0&0&0&1&0&0&0\\
0&0&0&-&0&-&-&1&0&0&1&0&0&1&0&0&1&0&0&0&0&0&0&-&0&0&-&0&0&-&0&0&-&0&0&0&0&0&0&1\\
0&0&-&0&-&0&-&-&0&1&0&0&1&0&0&0&0&0&0&-&0&0&-&0&0&-&0&0&-&0&0&0&0&0&0&1&0&0&1&0\\
0&-&0&0&1&-&0&-&1&0&0&0&0&0&0&-&0&0&-&0&0&-&0&0&-&0&0&0&0&0&0&1&0&0&1&0&0&1&0&0\\
-&0&0&0&1&1&-&0&0&0&0&-&0&0&-&0&0&-&0&0&-&0&0&0&0&0&0&1&0&0&1&0&0&1&0&0&1&0&0&0\\
0&0&0&-&0&0&-&0&0&-&-&1&0&0&0&1&0&1&0&0&0&-&0&0&0&0&1&0&-&0&0&0&0&0&0&-&1&0&0&0\\
0&0&-&0&0&-&0&0&-&0&-&-&0&0&1&0&1&0&0&0&-&0&0&0&0&1&0&0&0&0&0&1&0&0&-&0&0&0&0&-\\
0&-&0&0&-&0&0&0&1&-&0&-&0&1&0&0&0&0&0&-&0&0&0&1&1&0&0&0&0&0&1&0&0&-&0&0&0&0&-&0\\
-&0&0&0&0&0&0&1&1&1&-&0&1&0&0&0&0&0&-&0&0&0&1&0&0&0&0&-&0&1&0&0&-&0&0&0&0&-&0&0\\
0&0&0&-&0&-&0&0&0&0&0&-&0&-&-&1&0&0&1&0&1&0&0&0&-&0&0&0&0&1&0&0&0&0&0&1&0&0&-&0\\
0&0&-&0&-&0&0&0&0&0&-&0&-&0&-&-&0&1&0&0&0&0&0&-&0&0&0&1&1&0&0&0&0&0&1&0&0&-&0&0\\
0&-&0&0&0&0&0&1&0&-&0&0&1&-&0&-&1&0&0&0&0&0&-&0&0&0&1&0&0&0&0&-&0&1&0&0&-&0&0&0\\
-&0&0&0&0&0&1&0&-&0&0&0&1&1&-&0&0&0&0&-&0&-&0&0&0&1&0&0&0&0&-&0&1&0&0&0&0&0&0&1\\
0&0&0&-&-&0&0&0&0&-&0&0&0&0&-&0&0&-&-&1&0&1&0&0&0&0&0&-&0&0&0&1&1&0&0&0&0&0&1&0\\
0&0&-&0&0&0&0&1&-&0&0&0&0&-&0&0&-&0&-&-&1&0&0&0&0&0&-&0&0&0&1&0&0&0&0&-&0&1&0&0\\
0&-&0&0&0&0&1&0&0&0&0&1&-&0&0&0&1&-&0&-&0&0&0&-&0&-&0&0&0&1&0&0&0&0&-&0&1&0&0&0\\
-&0&0&0&0&1&0&0&0&0&1&0&0&0&0&1&1&1&-&0&0&0&-&0&-&0&0&0&1&0&0&0&0&-&0&0&0&0&0&-\\
0&0&0&-&0&0&0&1&0&1&0&0&-&0&0&0&0&-&0&0&0&-&-&1&1&0&0&0&0&0&-&0&0&0&1&0&0&0&0&-\\
0&0&-&0&0&0&1&0&1&0&0&0&0&0&0&1&-&0&0&0&-&0&-&-&0&0&0&-&0&-&0&0&0&1&0&0&0&0&-&0\\
0&-&0&0&0&1&0&0&0&0&0&-&0&0&1&0&0&0&0&1&1&-&0&-&0&0&-&0&-&0&0&0&1&0&0&0&0&-&0&0\\
-&0&0&0&1&0&0&0&0&0&-&0&0&1&0&0&0&0&1&0&1&1&-&0&0&-&0&0&0&0&0&1&0&0&0&-&-&0&0&0\\
0&0&0&-&0&0&1&0&0&0&-&0&1&0&0&0&0&0&0&1&-&0&0&0&0&-&-&1&0&0&0&-&0&-&0&0&0&1&0&0\\
0&0&-&0&0&1&0&0&0&-&0&0&0&0&0&-&0&0&1&0&0&0&0&1&-&0&-&-&0&0&-&0&-&0&0&0&1&0&0&0\\
0&-&0&0&1&0&0&0&-&0&0&0&0&0&-&0&0&1&0&0&0&0&1&0&1&-&0&-&0&-&0&0&0&0&0&1&0&0&0&-\\
-&0&0&0&0&0&0&-&0&0&0&1&0&-&0&0&1&0&0&0&0&1&0&0&1&1&-&0&-&0&0&0&0&0&1&0&0&0&-&0\\
0&0&0&-&0&1&0&0&1&0&0&0&0&-&0&0&0&0&0&-&0&0&1&0&0&0&0&1&0&-&-&1&0&0&-&0&-&0&0&0\\
0&0&-&0&1&0&0&0&0&0&0&-&-&0&0&0&0&0&-&0&0&1&0&0&0&0&1&0&-&0&-&-&0&-&0&0&0&0&0&1\\
0&-&0&0&0&0&0&-&0&0&-&0&0&0&0&1&0&-&0&0&1&0&0&0&0&1&0&0&1&-&0&-&-&0&0&0&0&0&1&0\\
-&0&0&0&0&0&-&0&0&-&0&0&0&0&1&0&-&0&0&0&0&0&0&-&1&0&0&0&1&1&-&0&0&0&0&1&0&1&0&0\\
0&0&0&-&1&0&0&0&0&0&0&1&0&0&0&-&-&0&0&0&0&0&-&0&0&1&0&0&0&0&1&0&0&-&-&1&0&-&0&0\\
0&0&-&0&0&0&0&-&0&0&1&0&0&0&-&0&0&0&0&1&0&-&0&0&1&0&0&0&0&1&0&0&-&0&-&-&-&0&0&0\\
0&-&0&0&0&0&-&0&0&1&0&0&0&-&0&0&0&0&1&0&-&0&0&0&0&0&0&-&1&0&0&0&1&-&0&-&0&0&0&1\\
-&0&0&0&0&-&0&0&1&0&0&0&-&0&0&0&0&1&0&0&0&0&0&1&0&0&-&0&0&0&0&-&1&1&-&0&0&0&1&0\\
0&0&0&-&0&0&0&-&-&0&0&0&0&0&1&0&0&0&-&0&0&0&0&1&0&-&0&0&1&0&0&0&0&1&0&0&0&-&-&1\\
0&0&-&0&0&0&-&0&0&0&0&1&0&1&0&0&0&-&0&0&0&0&1&0&-&0&0&0&0&0&0&-&1&0&0&0&-&0&-&-\\
0&-&0&0&0&-&0&0&0&0&1&0&1&0&0&0&-&0&0&0&0&1&0&0&0&0&0&1&0&0&-&0&0&0&0&-&1&-&0&-\\
-&0&0&0&-&0&0&0&0&1&0&0&0&0&0&-&0&0&0&1&1&0&0&0&0&0&1&0&0&-&0&0&0&0&-&0&1&1&-&0\\
  \end{array}\right)
 \]
 \end{tiny}
 \caption{A quasi-balanced $\w(40,12)$.}
 \label{signed-w-40-12}
\end{figure}

\begin{figure}[H]
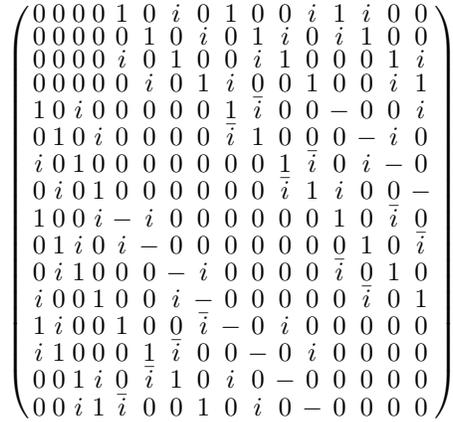

 \[
  \arraycolsep=1.25pt\def\arraystretch{0.625}
 \left(\begin{array}{cccccccccccccccc}
0 &0 &0 &0 &1 &0 &i &0 &1 &0 &0 &i &1 &i &0 &0 \\
0 &0 &0 &0 &0 &1 &0 &i &0 &1 &i &0 &i &1 &0 &0 \\
0 &0 &0 &0 &i &0 &1 &0 &0 &i &1 &0 &0 &0 &1 &i \\
0 &0 &0 &0 &0 &i &0 &1 &i &0 &0 &1 &0 &0 &i &1 \\
1 &0 &i &0 &0 &0 &0 &0 &1 &\bar i &0 &0 &- &0 &0 &i \\
0 &1 &0 &i &0 &0 &0 &0 &\bar i &1 &0 &0 &0 &- &i &0 \\
i &0 &1 &0 &0 &0 &0 &0 &0 &0 &1 &\bar i &0 &i &- &0 \\
0 &i &0 &1 &0 &0 &0 &0 &0 &0 &\bar i &1 &i &0 &0 &- \\
1 &0 &0 &i &- &i &0 &0 &0 &0 &0 &0 &1 &0 &\bar i &0 \\
0 &1 &i &0 &i &- &0 &0 &0 &0 &0 &0 &0 &1 &0 &\bar i \\
0 &i &1 &0 &0 &0 &- &i &0 &0 &0 &0 &\bar i &0 &1 &0 \\
i &0 &0 &1 &0 &0 &i &- &0 &0 &0 &0 &0 &\bar i &0 &1 \\
1 &i &0 &0 &1 &0 &0 &\bar i &- &0 &i &0 &0 &0 &0 &0 \\
i &1 &0 &0 &0 &1 &\bar i &0 &0 &- &0 &i &0 &0 &0 &0 \\
0 &0 &1 &i &0 &\bar i &1 &0 &i &0 &- &0 &0 &0 &0 &0 \\
0 &0 &i &1 &\bar i &0 &0 &1 &0 &i &0 &- &0 &0 &0 &0 \\
\end{array}\right)
 \]
 \caption{Quasi-balanced signing of an $\srg(16,6,2,2)$ over $\R_4$.}
 \label{srg-16-6-2-2}
\end{figure}

\begin{figure}[H]
 \[
 \arraycolsep=1.25pt\def\arraystretch{0.625}
 \left(
 \begin{array}{cccccccccccccccc}
0&1&1&1&1&1&0&0&0&0&0&0&0&0&0&0\\
1&0&0&0&0&0&1&1&1&1&0&0&0&0&0&0\\
1&0&0&0&0&0&-&0&0&0&1&1&1&0&0&0\\
1&0&0&0&0&0&0&-&0&0&-&0&0&1&1&0\\
1&0&0&0&0&0&0&0&-&0&0&-&0&-&0&1\\
1&0&0&0&0&0&0&0&0&-&0&0&-&0&-&-\\
0&1&-&0&0&0&0&0&0&0&0&0&0&1&-&1\\
0&1&0&-&0&0&0&0&0&0&0&-&1&0&0&-\\
0&1&0&0&-&0&0&0&0&0&1&0&-&0&1&0\\
0&1&0&0&0&-&0&0&0&0&-&1&0&-&0&0\\
0&0&1&-&0&0&0&0&1&-&0&0&0&0&0&1\\
0&0&1&0&-&0&0&-&0&1&0&0&0&0&-&0\\
0&0&1&0&0&-&0&1&-&0&0&0&0&1&0&0\\
0&0&0&1&-&0&1&0&0&-&0&0&1&0&0&0\\
0&0&0&1&0&-&-&0&1&0&0&-&0&0&0&0\\
0&0&0&0&1&-&1&-&0&0&1&0&0&0&0&0
 \end{array}
 \right)
\]
\caption{An srg-balanced signing of $\srg(16,5,0,2)$.}
\label{srg-16-5-0-2}
\end{figure}

\begin{figure}[H]
 \[
 \arraycolsep=1.25pt\def\arraystretch{0.625}
 \left(
 \begin{array}{cccccccccccccccc}
0&0&0&0&0&0&0&1&1&1&1&1&1&1&1&1\\
0&0&0&0&1&1&1&0&0&0&-&-&-&1&1&1\\
0&0&0&0&1&1&1&1&1&1&0&0&0&-&-&-\\
0&0&0&0&1&1&1&-&-&-&1&1&1&0&0&0\\
0&1&1&1&0&0&0&0&-&1&0&-&1&0&-&1\\
0&1&1&1&0&0&0&1&0&-&1&0&-&1&0&-\\
0&1&1&1&0&0&0&-&1&0&-&1&0&-&1&0\\
1&0&-&1&0&-&1&0&0&0&0&-&1&0&1&-\\
1&0&-&1&1&0&-&0&0&0&1&0&-&-&0&1\\
1&0&-&1&-&1&0&0&0&0&-&1&0&1&-&0\\
1&1&0&-&0&-&1&0&1&-&0&0&0&0&-&1\\
1&1&0&-&1&0&-&-&0&1&0&0&0&1&0&-\\
1&1&0&-&-&1&0&1&-&0&0&0&0&-&1&0\\
1&-&1&0&0&-&1&0&-&1&0&1&-&0&0&0\\
1&-&1&0&1&0&-&1&0&-&-&0&1&0&0&0\\
1&-&1&0&-&1&0&-&1&0&1&-&0&0&0&0
 \end{array}
 \right)
\]
\caption{An srg-balanced signing of an $\srg(16,9,4,6)$.}
\label{sgr-16-9-4-6}
\end{figure}

\begin{figure}[H]
\begin{small}
 \[
 \arraycolsep=1.25pt\def\arraystretch{0.625}
 \left(
 \begin{array}{cccccccccccccccccccccccccccc}
0&0&0&0&0&0&0&0&0&0&0&0&0&0&0&0&1&1&1&1&1&1&1&1&1&1&1&1\\
0&0&1&1&1&1&0&1&1&0&0&1&1&0&0&0&-&-&0&0&0&0&1&1&0&0&0&0\\
0&1&0&-&1&0&1&1&0&1&0&1&0&1&0&0&1&0&-&0&0&0&-&0&1&0&0&0\\
0&1&1&0&0&-&-&0&-&-&0&0&1&1&0&0&0&-&1&0&0&0&0&-&1&0&0&0\\
0&1&-&0&0&-&1&1&0&0&1&-&0&0&1&0&-&0&0&-&0&0&1&0&0&1&0&0\\
0&1&0&1&1&0&-&0&1&0&1&0&-&0&1&0&0&1&0&1&0&0&0&-&0&-&0&0\\
0&0&1&-&1&-&0&0&0&-&-&0&0&-&1&0&0&0&-&1&0&0&0&0&-&1&0&0\\
0&1&-&0&-&0&0&0&1&-&-&1&0&0&0&1&-&0&0&0&1&0&-&0&0&0&1&0\\
0&1&0&1&0&-&0&-&0&1&-&0&-&0&0&-&0&-&0&0&1&0&0&1&0&0&-&0\\
0&0&1&-&0&0&1&-&1&0&1&0&0&1&0&1&0&0&1&0&1&0&0&0&-&0&-&0\\
0&0&0&0&1&1&1&-&-&-&0&0&0&0&1&-&0&0&0&-&1&0&0&0&0&-&1&0\\
0&1&-&0&1&0&0&-&0&0&0&0&1&-&-&1&1&0&0&0&0&-&1&0&0&0&0&-\\
0&1&0&-&0&1&0&0&1&0&0&-&0&-&-&-&0&-&0&0&0&1&0&-&0&0&0&1\\
0&0&1&1&0&0&1&0&0&-&0&-&-&0&-&1&0&0&-&0&0&-&0&0&1&0&0&1\\
0&0&0&0&1&1&-&0&0&0&-&-&-&1&0&1&0&0&0&-&0&1&0&0&0&1&0&-\\
0&0&0&0&0&0&0&1&-&-&1&1&-&-&-&0&0&0&0&0&1&1&0&0&0&0&-&-\\
1&1&1&0&-&0&0&-&0&0&0&1&0&0&0&0&0&1&-&-&-&1&1&0&0&0&0&0\\
1&1&0&-&0&1&0&0&-&0&0&0&-&0&0&0&-&0&1&1&-&-&0&1&0&0&0&0\\
1&0&-&1&0&0&1&0&0&-&0&0&0&1&0&0&1&-&0&1&-&1&0&0&-&0&0&0\\
1&0&0&0&-&1&-&0&0&0&1&0&0&0&1&0&1&-&-&0&1&-&0&0&0&1&0&0\\
1&0&0&0&0&0&0&1&1&-&-&0&0&0&0&-&1&1&1&-&0&-&0&0&0&0&-&0\\
1&0&0&0&0&0&0&0&0&0&0&-&1&1&-&-&-&1&-&1&1&0&0&0&0&0&0&-\\
1&-&-&0&1&0&0&-&0&0&0&1&0&0&0&0&-&0&0&0&0&0&0&-&1&1&-&1\\
1&-&0&-&0&-&0&0&1&0&0&0&-&0&0&0&0&-&0&0&0&0&1&0&1&-&1&-\\
1&0&1&1&0&0&1&0&0&1&0&0&0&-&0&0&0&0&1&0&0&0&-&-&0&1&1&-\\
1&0&0&0&1&-&-&0&0&0&1&0&0&0&-&0&0&0&0&-&0&0&-&1&-&0&1&1\\
1&0&0&0&0&0&0&1&-&1&-&0&0&0&0&1&0&0&0&0&1&0&1&-&-&-&0&1\\
1&0&0&0&0&0&0&0&0&0&0&-&1&-&1&1&0&0&0&0&0&1&-&1&1&-&-&0
 \end{array}
 \right)
\]
\end{small}
\caption{An srg-balanced signing of an $\srg(28,12,6,4)$}
\label{srg-28-12-6-4}
\end{figure}

\end{document}